\documentclass[a4paper]{article}

\usepackage{a4}
\usepackage{amsmath}
\usepackage{amssymb}
\usepackage{amsfonts}
\usepackage{amsthm}
\usepackage{geometry}
\usepackage{stmaryrd}

\usepackage[textsize=footnotesize]{todonotes}

\usepackage{hyperref}
\hypersetup{
    pdftitle={An L2-quotient algorithm for finitely presented groups on 
        arbitrarily many generators},
    pdfauthor={Sebastian Jambor},
    colorlinks=false,
    pdfborder={0 0 0}
}

\newtheorem{thm}{Theorem}[section]
\newtheorem{lemma}[thm]{Lemma}
\newtheorem{prop}[thm]{Proposition}
\newtheorem{cor}[thm]{Corollary}
\theoremstyle{definition}
\newtheorem{definition}[thm]{Definition}

\newtheorem{remark}[thm]{Remark}
\newtheorem{example}[thm]{Example}
\newtheorem{notation}[thm]{Notation}

\newtheoremstyle{alg}
    {3pt}
    {3pt}
    {}
    {}
    {\bfseries}
    {.}
    {\newline}
    {}

\theoremstyle{alg}
\newtheorem{algorithm}[thm]{Algorithm}

\evensidemargin\oddsidemargin
\DeclareMathOperator{\Ln}{L}
\DeclareMathOperator{\SL}{SL}
\DeclareMathOperator{\PSL}{PSL}
\DeclareMathOperator{\GL}{GL}
\DeclareMathOperator{\PGL}{PGL}
\DeclareMathOperator{\Aut}{Aut}
\DeclareMathOperator{\Alt}{A}
\DeclareMathOperator{\Sym}{S}
\DeclareMathOperator{\Stab}{Stab}
\DeclareMathOperator{\Gal}{Gal}
\DeclareMathOperator{\PGammaL}{P\Gamma L}
\DeclareMathOperator{\GammaL}{\Gamma L}
\DeclareMathOperator{\tr}{tr}
\DeclareMathOperator{\Norm}{N}
\DeclareMathOperator{\Spec}{Spec}
\DeclareMathOperator{\MaxSpec}{MaxSpec}
\DeclareMathOperator{\MinAss}{MinAss}
\DeclareMathOperator{\im}{im}
\DeclareMathOperator{\V}{V}
\DeclareMathOperator{\I}{I}
\DeclareMathOperator{\Die}{D}
\DeclareMathOperator{\Nor}{N}
\DeclareMathOperator{\Cyc}{C}
\def\tri{\trianglelefteq}
\def\veps{\varepsilon}

\newcommand{\N}{{\mathbb{N}}}
\newcommand{\Z}{{\mathbb{Z}}}
\newcommand{\Q}{{\mathbb{Q}}}
\newcommand{\F}{{\mathbb{F}}}
\newcommand{\C}{{\mathbb{C}}}
\newcommand{\mcC}{{\mathcal{C}}}
\newcommand{\mcE}{{\mathcal{E}}}
\newcommand{\mcI}{{\mathcal{I}}}
\newcommand{\mcR}{{\mathcal{R}}}
\newcommand{\mcP}{{\mathcal{P}}}
\newcommand{\mfA}{{\mathfrak{A}}}
\newcommand{\mfD}{{\mathfrak{D}}}

\newcommand{\mfS}{{\mathfrak{S}}}

\newcommand{\ol}[1]{\overline{#1}}

\newcommand{\wt}[1]{\widetilde{#1}}
\newcommand{\wh}[1]{\widehat{#1}}

\date{}
\author{Sebastian Jambor}
\title{An $\Ln_2$-quotient algorithm for finitely presented groups \\
    on arbitrarily many generators}

\begin{document}

\maketitle

\begin{abstract}
\noindent
\textbf{Abstract.}
We generalize the Plesken-Fabia{\'n}ska $\Ln_2$-quotient algorithm for finitely 
presented groups on two or three generators to allow an arbitrary number of 
generators.
The main difficulty lies in a constructive description of the invariant ring 
of $\GL(2, K)$ on $m$ copies of $\SL(2, K)$ by simultaneous conjugation.
By giving this description, we generalize and simplify some of the known 
results in invariant theory.  An implementation of the algorithm is available 
in the computer algebra system \textsc{Magma}.


\end{abstract}

\let\thefootnote\relax\footnote{The author was supported by the Alexander von 
Humboldt Foundation via a Feodor Lynen Research Fellowship}

\section{Introduction}

The Plesken-Fabia{\'n}ska $\Ln_2$-quotient algorithm \cite{l2q} takes as input 
a finitely presented group $G$ on two generators and computes all quotients of 
$G$ which are isomorphic to $\PSL(2, q)$ or $\PGL(2, q)$.
The algorithm finds all possible prime powers~$q$, and also deals with the case 
when there are infinitely many.
This was adapted by Fabia{\'n}ska \cite{fabianska} to allow finitely presented 
groups on three generators. In particular, the algorithm can decide whether $G$
has infinitely many quotients isomorphic to $\PSL(2, q)$ or $\PGL(2, q)$, so in 
some cases it can be used to prove that a finitely presented group is infinite.
This has been applied for example in \cite{havas}.
In this paper, we generalize the algorithm to allow finitely presented groups 
on an arbitrary number of generators.

The method of Fabia{\'n}ska and Plesken uses the character of representations 
$F_2 \to \SL(2, K)$, where $F_2$ is the free group of rank~$2$ and $K$ is an 
arbitrary field. The character is fully determined by the traces of the images 
of the two generators of $F_2$ and their product. This observation goes as far 
back as to Vogt \cite{vogt} and Fricke and Klein \cite{fricke}.
Horowitz \cite{horowitz} gives a rigorous proof of this fact, and generalizes 
it to representations $F_m \to \SL(2, K)$ for an arbitrary $m$, by proving that 
a character is fully determined by $2^m - 1$ traces.
While the traces for $m = 2$ are algebraically independent (that is, for every 
choice of traces for the images of the two generators and their product, there 
always exists a representation with these traces), this is no longer true for 
$m > 2$. The problem is thus to describe all relations between the traces, or 
equivalently, to give a presentation for the invariant ring 
$K[\SL(2,K)^m]^{\GL(2, K)}$, where $\GL(2, K)$ acts on $m$ copies of 
$\SL(2, K)$ by conjugation. Furthermore, we need this description to be 
independent of the characteristic of the field $K$. This problem has a long 
history. Procesi \cite{procesi} proves that the invariant ring 
$K[(K^{n \times n})^m]^{\GL(n, K)}$ is finitely generated if $K$ has 
characteristic zero, and Donkin \cite{donkin} generalizes this to arbitrary 
fields~$K$. However, their results are non-constructive.
Procesi \cite{procesi2x2} gives an implicit description of the invariant ring 
$\C[(\C^{2 \times 2})^m]^{\GL(2, \C)}$, and Drensky \cite{drensky} gives an 
explicit description, however, their results are not valid for fields of 
characteristic~$2$. Magnus \cite{magnus} uses Horowitz's results to to give a 
description of the quotient ring of the invariant ring.

We will use the approach of Horowitz and Magnus to get a partial description of 
the invariant ring. The methods are constructive and the arguments are shorter 
than the original arguments; at the same time we get more precise results, 
needed for the algorithm. This theory is developed in Section~\ref{S:fricke}.
Sections~\ref{S:tracetuples}--\ref{S:pglpsl} are adaptations of \cite{l2q}, 
where we have to generalize results on characters and traces to work for 
arbitrarily many generators. 
Up until the end of Section~\ref{S:pglpsl}, all results assume that 
representations restricted to the subgroup generated by the first two 
generators is absolutely irreducible. The results in Section~\ref{S:arbitrary} 
show how the general case can be reduced to this special case. 
In Section~\ref{S:subgrouptests}, a new test to recognize epimorphisms onto 
$\Alt_4$, $\Sym_4$, and~$\Alt_5$ is developed, since the test described in 
\cite{l2q} is inefficient for more than two generators. 
Section~\ref{S:l2ideals} describes the proper notation and theory to deal with 
an infinite number of $\Ln_2$-quotients.
The algorithm is given in Section~\ref{S:algorithm}, with several examples in 
Section~\ref{S:examples}.

\section{Fricke characters}
\label{S:fricke}

Througout the paper, $K$ is an arbitrary field and $m \geq 2$ an integer, 
unless specified otherwise. In this section, we adopt the following notation.

\begin{notation}
Given matrices $A_1, \dotsc, A_m \in \SL(2, K)$ and a list 
$i_1, \dotsc, i_k \in \{\pm 1, \dotsc, \pm m\}$, we set 
$t_{i_1,\dotsc,i_k} := \tr(A_{i_1}\dotsb A_{i_k})$, where $A_{-i} := A_i^{-1}$ 
for $i \in \{1, \dotsc, m\}$.
If $I = \{i_1, \dotsc, i_k\} \subseteq \{1, \dotsc, m\}$ with 
$i_1 < i_2 < \dotsb < i_k$, then $t_I := t_{i_1, \dotsc, i_k}$.
\end{notation}

Let $A_1, A_2, A_3 \in \SL(2, K)$. The traces satisfy the following basic 
identities.
\begin{align}
    t_{1,1,2} & = t_1t_{1,2} - t_2, \\
    t_{-1,2} & = t_1t_2 - t_{1,2}, \\
    t_{1,2,1,3} & = t_{1,2}t_{1,3} + t_{2,3} - t_2t_3, \label{E:t1213} \\
    t_{1,3,2} & = -t_{1,2,3} + t_1t_{2,3} + t_2t_{1,3} 
        + t_3t_{1,2} - t_1t_2t_3. \label{E:t132}
\end{align}
The first two identities are easy consequences of the Cayley-Hamilton Theorem,
and the others are easy consequences of the first two
(for \eqref{E:t1213} consider $\tr((A_1A_2)^2(A_2^{-1}A_3))$; 
for \eqref{E:t132} consider 
$\tr(A_1^{-1}(A_2^{-1}A_3)) = \tr((A_2A_1)^{-1}A_3)$).

We first prove that all traces of words in the $A_i$ are consequences of the 
$t_I$ with $\emptyset \neq I \subseteq \{1, \dotsc, m\}$. This was already 
observed by Vogt \cite{vogt} and later by Fricke and Klein \cite{fricke}.
The first rigorous proof of this fact was given by Horowitz \cite{horowitz}, 
and a shorter proof by Fabia{\'n}ska and Plesken \cite{l2q}.

Let $F_m$ be the free group of rank $m$, generated by $g_1, \dotsc, g_m$.

\begin{thm}[{\cite[Theorem 3.1]{horowitz}}, {\cite[Lemma 2.1]{l2q}}]
\label{T:tau}
Let $X_m := \{x_I \mid \emptyset \neq I \subseteq \{1, \dotsc, m\}\}$ be a set 
of indeterminates over~$\Z$. For every $w \in F_m$ there exists a polynomial 
$\tau(w) \in \Z[X_m]$, such that for every field $K$ and every $m$-tuple 
$A = (A_1, \dotsc, A_m) \in \SL(2, K)^m$,
\[
    \tr(w(A_1, \dotsc, A_m)) = \veps_A(\tau(w)),
\]
where $\veps_A\colon \Z[X_m] \to K$ is the evaluation map which sends $x_I$
to~$t_I$.
\end{thm}
Since the proof in \cite{horowitz} is lengthy, and the result in \cite{l2q} is 
not as general, we present a short proof here in its entirety. The basic idea 
is that of \cite[Lemma 2.1]{l2q}.
\begin{proof}
We assume that $w$ is freely and cyclically reduced and proceed by induction on 
the length of $w$.
If $w$ is conjugate to $g_i^{-1}w'$ for some $w' \in F_n$ of length $|w|-1$,
set $\tau(w) = \tau(g_iw') - \tau(g_i)\tau(w')$. Thus we may assume that all 
exponents of $w$ are positive.
If $w$ is conjugate to $g_iw'g_iw''$ for some $w', w'' \in F_m$ with 
$|w'| + |w''| = |w| - 2$, set 
$\tau(w) = \tau(g_iw')\tau(g_iw'') + \tau(w'w'') - \tau(w')\tau(w'')$.
We are left to deal with the case where $w$ is of the form 
$w = g_{i_1}\dotsb g_{i_k}$ where the $i_j$ are pairwise distinct. We may 
assume $i_1 < i_j$ for all $j \in \{2, \dotsc, k\}$. The case 
$i_1 < \dotsb < i_k$ is the induction basis, so there is nothing to do. 
Otherwise, let $j$ be the smallest index with $i_j > i_{j+1}$.
Set $w_1 := g_{i_1}\dotsb g_{i_{j-1}}$, $w_2 := g_{i_j}$, and 
$w_3 := g_{i_{j+1}}\dotsb g_{i_k}$, so $w = w_1w_2w_3$.
By equation~\eqref{E:t132} we may set 
$\tau(w) := -\tau(w_1w_3w_2) + \tau(w_1)\tau(w_2w_3) + \tau(w_2)\tau(w_1w_3) 
    + \tau(w_3)\tau(w_1w_2) - \tau(w_1)\tau(w_2)\tau(w_3)$.
Either $w_1w_3w_2$ is of the desired form, or we repeat this process. This 
terminates after finitely many steps.
\end{proof}

We call $\tau(w)$ the \textit{trace polynomial} of $w$. If $n > 2$, then 
$\tau(w)$ is not unique. For example, define the \textit{Fricke polynomial}
\begin{multline*}
    \phi(x_1,x_2,x_3,x_{12}, x_{13}, x_{23}, x_{123}) := x_{123}^2 + 
        (x_1x_2x_3 - x_1x_{23} - x_2x_{13} - x_3x_{12})x_{123} \\
    + x_1^2 + x_2^2 + x_3^2 + x_{12}^2 + x_{13}^2 + x_{23}^2 - x_1x_2x_{12} 
        - x_1x_3x_{13} - x_2x_3x_{23} + x_{12}x_{13}x_{23} - 4.
\end{multline*}
Then $\veps_A(\phi) = 0$ for every choice of $A$. Proofs appear for example in 
\cite[Section~2]{horowitz} and \cite[Lemma~2.2]{magnus}.
We will see below that $\phi$ is simply a determinant condition (see 
Proposition~\ref{P:t123} and Corollary~\ref{C:fricke}).

A lot of effort has been put into describing all polynomial relations between 
the traces. More precisely, let 
\[
    I_m := \{f \in \Z[X_m] \mid \veps_A(f) = 0 \text{ for all } 
        A_1, \dotsc, A_n \in \SL(2, \C)\}
\]
and $\Phi_m := \Z[X_m]/I_m$, the \textit{ring of Fricke characters}. 
It is easy to see that $\veps_A(f) = 0$ for all $A \in \SL(2, K)^m$, so the 
role of $\C$ is not special. Horowitz \cite[Theorem~4.3]{horowitz} proves 
$I_3 = \langle \phi \rangle$, and Whittemore \cite[Theorem~1]{whittemore} 
proves that $I_m$ is not principal if $m \geq 4$.
Magnus \cite[Theorem~2.1]{magnus} shows that $\Phi_m$ can be embedded into a 
finitely generated extension field of $\Q$ of transcendence degree~$3m-3$.
Note that $\Phi_m\otimes_{\Z}\C$ is isomorphic to the invariant ring 
$\C[\SL(2, \C)^m]^{\GL(2,\C)}$.
Procesi \cite{procesi2x2} gives a description of the invariant ring 
$\C[(\C^{2\times 2})^m]^{\GL(2, \C)}$, and an explicit presentation of the 
invariant ring with generators and relations is given by Drensky 
\cite[Theorem~2.3]{drensky}. However, these results are not valid for fields of 
characteristic~$2$, and hence cannot be applied to describe $\Phi_m$.

Our first aim is to partially describe $\Phi_m$; we give a presentation of a 
localisation of~$\Phi_m$, which will be enough for our algorithmic 
applications. By doing that, we will also find new and shorter proofs of some 
of the results mentioned above.

We will use the following basic result.

\begin{prop}[{\cite[Theorem~2]{macbeath}}, {\cite[Equation~(2.7)]{magnus}}, 
    {\cite[Proposition~4.1]{brumfiel}}, {\cite[Proposition~3.1]{l2q}}]
\label{P:l2:rho}
Let $A = (A_1, A_2) \in \SL(2, K)^2$.
Then $\langle A_1, A_2 \rangle$ is absolutely irreducible if and only if 
$(t_1, t_2, t_{12})$ is not a zero of 
\[
    \rho := x_1^2 + x_2^2 + x_{12}^2 - x_1x_2x_{12} - 4.
\]
\end{prop}

This is based on the fact that $\langle A_1, A_2 \rangle$ is absolutely 
irreducible if and only if $(I_2, A_1, A_2, A_1A_2)$ is a $K$-basis 
of~$K^{2 \times 2}$ (see for example \cite{l2q}), a result which we will also 
use several times.

The main result in this section shows that two matrices $A_1, A_2$ uniquely 
determine an arbitrary matrix by the specification of four traces; it also 
shows that the Fricke polynomial is really a determinant condition.
The basic idea of the proof has already been used by Brumfiel and Hilden 
\cite[Proposition~B.4]{brumfiel}.

\begin{prop}
\label{P:t123}
Let $A_1, A_2 \in \SL(2, K)$ such that $\langle A_1, A_2 \rangle$ is absolutely 
irreducible, and let $i \geq 3$. 
Given $T_i$, $T_{1i}$, $T_{2i}$, $T_{12i} \in K$, there exists a unique 
$A_i \in K^{2 \times 2}$ such that $t_I = T_I$ for all 
$I \in \{\{i\}, \{1,i\}, \{2,i\}, \{1,2,i\}\}$.
Moreover, $\det(A_i) = 1$ if and only if $\phi(t_1, \dotsc, t_{12i}) = 0$.

More precisely, let
\begin{align*}
\lambda_{0}^{i} & := (x_1^2 + x_2^2 + x_{12}^2 - x_1x_2x_{12} - 2)x_i 
    - x_1x_{1i} - x_2x_{2i} + (x_1x_2-x_{12})x_{12i},\\
\lambda_{1}^{i} & := -x_1x_i - x_2x_{12i} + x_{12}x_{2i} + 2x_{1i}, \\
\lambda_{2}^{i} & := -x_2x_i - x_1x_{12i} + x_{12}x_{1i} + 2x_{2i}, \\
\lambda_{12}^{i} & := -x_1x_{2i} - x_2x_{1i} - x_ix_{12} + x_1x_2x_i 
    + 2x_{12i};
\end{align*}
set 
$\Lambda_I := \lambda_{I}^i(t_1, t_2, T_i, t_{12}, T_{1i}, T_{2i}, T_{12i})$ 
for $I \in \{\{0\},\{1\},\{2\},\{1,2\}\}$.
Then
\[
    A_i = \frac{1}{\rho(t_1, t_2, t_{12})}(\Lambda_0 I_2 + \Lambda_1 A_1 
        + \Lambda_2 A_2 + \Lambda_{12}A_1A_2).
\]
\end{prop}
\begin{proof}
Since $\langle A_1, A_2 \rangle$ is absolutely irreducible, 
$(I_2, A_1, A_2, A_1A_2)$ is a $K$-basis of $K^{2 \times 2}$.
Thus if $A_i$ exists as in the statement, then 
$A_i = \mu_0 I_2 + \mu_1 A_1 + \mu_2 A_2 + \mu_{12}A_1A_2$ for some 
$\mu_i \in K$. 
Multiplying the equation from the left by the matrices $I_2, A_1, A_2, A_1A_2$ 
and taking traces shows that the $\mu_i$ are the unique solution of
\[
\begin{pmatrix}
    2      & t_1              & t_2              & t_{12}           \\
    t_1    & t_1^2-2          & t_{12}           & t_1t_{12} - t_2  \\
    t_2    & t_{12}           & t_2^2-2          & t_2t_{12} - t_1  \\
    t_{12} & t_1t_{12} - t_2 & t_2t_{12} - t_1   & t_{12}^2-2       
\end{pmatrix}
\begin{pmatrix} \mu_0 \\ \mu_1 \\ \mu_2 \\ \mu_{12} \end{pmatrix}
= 
\begin{pmatrix}
T_i \\
T_{1i} \\
T_{2i} \\
T_{12i}
\end{pmatrix},
\]
which is given by $\mu_i = \Lambda_i/\rho(t_1, t_2, t_{12})$. This proves the 
uniqueness and existence of~$A_i$. It remains to show the determinant 
condition. We use the idea of \cite[Proposition~3.1]{l2q}.
Let $\alpha$ be a root of $X^2 - t_1X + 1$; by enlarging $K$ if necessary, we 
may assume $\alpha \in K$. Let $v_1 \in K^{2 \times 1}$ be an eigenvector
of $A_1$ with eigenvalue~$\alpha$. Set $v_2 := A_2v_1$, and let 
$M \in \GL(2, K)$ be the matrix with columns $v_1$ and $v_2$. Set 
$B_j := M^{-1}A_jM$ for $j \in \{1,2,i\}$. Then
\[
    B_1 = \begin{pmatrix}
       \alpha & t_2(\alpha - t_1) + t_{12} \\ 
       0 & t_1 - \alpha
    \end{pmatrix}
    \quad\text{and}\quad
    B_2 = \begin{pmatrix}
        0 & -1 \\
        1 & t_2
    \end{pmatrix}.
\]
Since 
$B_i = 1/\rho(t_1, t_2, t_{12})(\Lambda_0 I_2 + \Lambda_1 B_1 + \Lambda_2B_2 
    + \Lambda_{12}B_1B_2)$, 
\[
    \det(A_i) = \det(B_i) = 
        \frac{\phi(t_1, \dotsc, t_{12i}) + \rho(t_1, t_2, t_{12})}
            {\rho(t_1, t_2, t_{12})},
\]
which concludes the proof.
\end{proof}

\begin{cor}
\label{C:fricke}
Let $A_1, A_2, A_3 \in \SL(2, K)$. The $t_I$ satisfy the Fricke relation, that 
is, $\phi(t_1, \dotsc, t_{123}) = 0$.
\end{cor}
\begin{proof}
We may assume without loss of generality that $K$ is algebraically closed.
By Proposition~\ref{P:t123}, the statement is true for the Zariski-open subset 
\[
    U = \{(A_1, A_2, A_3) \in \SL(2, K)^3 \mid 
        \rho(\tr(A_1), \tr(A_2), \tr(A_1A_2)) \neq 0\},
\]
so by continuity, it is true for all elements in $\SL(2, K)^3$.
\end{proof}

The following is a generalization of \cite[Theorem~2.2]{magnus} and 
\cite[Proposition~3.1]{l2q}. Set
\[
    \mcI_m := \{\{i\} \mid 1 \leq i \leq m\} 
        \cup \{\{i,j\} \mid 1\leq i \leq 2, i < j \leq m\} 
        \cup \{\{1,2,k\} \mid 3 \leq k \leq m\}.
\]

\begin{cor}
\label{C:alg_ind}
Let $T_I \in K$ for $I \in \mcI_m$ 
such that 
\[
    \rho(T_1, T_2, T_{12}) \neq 0 \quad\text{and}\quad 
        \phi(T_1, T_2, T_k, T_{12}, T_{1k}, T_{2k}, T_{12k}) = 0 
        \text{ for all } 3 \leq k \leq m. 
\]
Let $L$ be the splitting field of $X^2 - T_1X + 1 \in K[X]$.
There exists $A = (A_1, \dotsc, A_m) \in \SL(2, L)^m$ such that $t_I = T_I$ for 
all $I \in \mcI_m$, and $A$ is unique up to conjugation by $\GL(2, L)$.

There exists $A \in \SL(2, K)^m$ such that $t_I = T_I$ for all $I \in \mcI_m$ 
if and only if $\rho(T_1, T_2, T_{12}) = \Norm_{L/K}(z)$ for some $z \in L$.
In this case, $A$ is unique up to conjugation by $\GL(2, K)$.
\end{cor}
\begin{proof}
By \cite[Proposition~3.1]{l2q}, there exists $A' := (A_1, A_2) \in \SL(2, L)^2$ 
with $t_1 = T_1$, $t_2 = T_2$, and $t_{12} = T_{12}$; furthermore, $A'$ is 
unique up to $L$-equivalence, and $A'$ exists in $\SL(2,K)^2$ if and only if 
$\rho(T_1, T_2, T_{12})$ is a norm, in which case $A'$ is unique up to 
$K$-equivalence. By Proposition~\ref{P:t123}, the choice of $A'$ and the $T_I$ 
uniquely determines the matrices $A_3, \dotsc, A_m$.
\end{proof}

This result implies that the traces $t_J$ with $J \not\in \mcI_m$ can be 
expressed in the traces $t_I$ with $I\in\mcI_m$ if 
$\rho(t_1, t_2, t_{12}) \neq 0$. The next result gives the precise formulae.

\begin{prop}
\label{P:linear_relations}
Let $A_1, \dotsc, A_n \in \SL(2, K)$.
Let $3 \leq i < n$ and $\emptyset \neq j \subseteq \{i+1, \dotsc, n\}$.
The tuple $t = (t_I \mid \emptyset \neq I \subseteq \{1, \dotsc, n\})$ is a 
zero of the polynomials
\begin{align*}
    x_{ij} \rho     & - (\lambda_{0}^{i}x_j + \lambda_{1}^{i}x_{1j} 
        + \lambda_{2}^{i}x_{2j} + \lambda_{12}^{i}x_{12j}), \\
    x_{1ij} \rho    & - (\lambda_{0}^{i}x_{1j} 
        + \lambda_{1}^{i}(x_1x_{1j} - x_j) + \lambda_{2}^{i}x_{12j} 
        + \lambda_{12}^{i}(x_1x_{12j} - x_{2j})), \\
    x_{2ij} \rho    & - (\lambda_{0}^{i}x_{2j} + \lambda_{1}^{i}(-x_{12j} 
        + x_1x_{2j} + x_2x_{1j} + x_jx_{12} - x_1x_2x_j) 
        + \lambda_{2}^{i}(x_2x_{2j} - x_j) \\
    & + \lambda_{12}^{i}(x_{12}x_{2j} - x_1x_j + x_{1j})), \\
    x_{12ij} \rho   & - (\lambda_{0}^{i} x_{12j} 
        + \lambda_{1}^{i} (x_{12}x_{1j} - x_2x_j + x_{2j}) 
        + \lambda_{2}^{i}(x_2x_{12j} - x_{1j}) + \lambda_{12}^{i}(x_{12}x_{12j} 
        - x_j)).
\end{align*}
\end{prop}
\begin{proof}
It is enough to prove the statement if $\rho \neq 0$ (see proof of 
Corollary~\ref{C:fricke}). By Proposition~\ref{P:t123}, we see 
$A_i = \Lambda_0I_2 + \Lambda_1A_1 + \Lambda_2A_2 + \Lambda_{12}A_1A_2$.
Multiplying from the right by $A_j$ and from the left by 
$I_2, A_1, A_2, A_1A_2$ and taking traces yields the result.
\end{proof}

For a ring $R$ and $r \in R$, let $R_r$ denote the localisation of $R$ at the 
set $\{1, r, r^2, \dotsc\}$. While we do not have an explicit description of 
the ring $\Phi_m$ of Fricke characters, we have one for a localisation 
of~$\Phi_m$.

\begin{cor}
\label{C:iso}
Let 
\[
    \Phi'_m := (\Z[x_I \mid I \in \mcI_m]/
        \langle \phi_{123}, \dotsc, \phi_{12m} \rangle)_\rho,
\]
where $\phi_{12i} := \phi(x_1, x_2, x_i, x_{12}, x_{1i}, x_{2i}, x_{12i})$ for 
$3 \leq i \leq m$ and $\rho := \rho(x_1, x_2, x_{12})$.
The ring homomorphism $\Phi_m' \to (\Phi_m)_\rho$ defined by $x_I \mapsto x_I$ 
is an isomorphism.
\end{cor}
\begin{proof}
Define a ring homomorphism 
$\alpha\colon \Z[x_I \mid I \in \mcI_m] \to (\Phi_m)_\rho$ by mapping $x_I$ 
to~$x_I$.
By Proposition~\ref{P:linear_relations}, $\alpha$ is surjective, and by 
Proposition~\ref{P:t123}, $\alpha$ factors over $\Phi_m'$. But 
Proposition~\ref{P:t123} also shows that the induced map is injective (see also 
Corollary~\ref{C:alg_ind}).
\end{proof}

\begin{cor}[{\cite[Theorem~2.1]{magnus}}]
The quotient field of $\Phi_m$ is isomorphic to a $(m-2)$-fold quadratic 
extension of a rational function field of trancendence degree $3m-3$ over $\Q$.
\end{cor}

\section{Trace tuples}
\label{S:tracetuples}

The ultimate goal is to get a bijection between prime ideals of $\Phi_m'$ and 
equivalence classes of representations $F_m \to \SL(2, K)$, where $K$ ranges 
over all fields.

\begin{definition}
A tuple $t = (t_I \mid I \in \mcI_m) \in K^{\mcI_m}$ is a \textit{trace tuple} 
if 
\[
    \rho(t_1, t_2, t_{12}) \neq 0 \quad\text{and}\quad
    \phi(t_1, t_2, t_i, t_{12}, t_{1i}, t_{2i}, t_{12i}) = 0 
        \text{ for all } 3 \leq i \leq m.
\]

If $A \in \SL(2, K)^m$ such that $\langle A_1, A_2 \rangle$ is absolutely 
irreducible, then the traces $t_I = \tr(A_{i_1}\dotsb A_{i_k})$ for 
$I = \{i_1 < \dotsb < i_k\} \in \mcI_m$ form a trace tuple. We call this the 
\textit{trace tuple of $A$}, and $A$ a \textit{realization of $t$}.
The tuple $(t_J \mid \emptyset \neq J \subseteq \{1, \dotsc, m\})$ is the 
\textit{full trace tuple of~$A$}.
\end{definition}

\begin{definition}
Let $\Gamma$ be a group generated by $\gamma_1, \dotsc, \gamma_m$.
Let
\[
    \mcR(\Gamma, K) := \{\Delta\colon \Gamma \to \SL(2, K) \mid 
        \Delta_{|\langle \gamma_1, \gamma_2 \rangle} 
        \text{ is absolutely irreducible}\}.
\]
\end{definition}

\begin{remark}
The set $\mcR(F_m, K)$ is in bijection to the set of matrices 
$A \in \SL(2, K)^m$ such that $\langle A_1, A_2 \rangle$ is absolutely 
irreducible, so we may talk about trace tuples of $\Delta$ and regard 
representations as realizations of trace tuples.
\end{remark}

We will first prove the results for finite fields and then generalize to 
arbitrary fields.

\subsection{Finite fields}

\begin{definition}
Let $t, t' \in \F_q^{\mcI_m}$ be trace tuples.
Let $L$ and $L'$ be the subfields of $\F_q$ generated by $t$ and~$t'$, 
respectively. We say that $t$ and $t'$ are \textit{equivalent} if there exists 
an isomorphism $\alpha \colon L \to L'$ such that $\alpha(t_I) = \alpha(t_I')$ 
for all $I \in \mcI_m$.
\end{definition}

\begin{remark}
By Corollary~\ref{C:alg_ind}, every trace tuple $t \in \F_q^{\mcI_m}$ has a 
realization $A \in \SL(2, \F_q)^m$.
\end{remark}

Let $t \in \F_q^{\mcI_m}$ be a trace tuple. Define a ring homomorphism 
$\alpha_t\colon \Phi_m' \to \F_q$ by $\alpha_t(x_I) := t_I$ for $I \in \mcI_m$.
Then $P_t := \ker(\alpha_t)$ is a maximal ideal of $\Phi_m'$.

Conversely, let $P  \in \MaxSpec(\Phi_m')$, where $\MaxSpec(\Phi_m')$ denotes 
the set of maximal ideals of $\Phi_m'$.
Let $\F_q = \Phi_m'/P$, and set $(t_P)_I := x_I + P \in \F_q$ for 
$I \in \mcI_m$.
Then $t_P := ((t_P)_I \mid I \in \mcI_m) \in \F_q^{\mcI_m}$ is a trace tuple.

\begin{thm}
\label{T:maxspec}
The maps $P \mapsto t_P$ and $t \mapsto P_t$ induce mutually inverse bijections 
between $\MaxSpec(\Phi_m')$ and the set of equivalence classes of trace tuples 
over finite fields.
\end{thm}
\begin{proof}
Let $P \in \MaxSpec(\Phi_m')$.  Since $\alpha_{t_P}(x_I) = x_I + P$ by 
definition, we see $P = P_{t_P}$. Now let $t \in \F_q^{\mcI_m}$ be a trace 
tuple; we may assume that $\F_q$ is generated by $t$.
Then $\Phi_m'/P_t$ is a field with $q$ elements.
Define a homomorphism $\F_q \to \Phi_m'/P_t$ by $t_I \mapsto x_I + P_t$.
By definition of $P_t$ this is well-defined and it is clearly surjective, hence 
an isomorphism; it maps $t$ to $t_{P_t}$, so $t$ is equivalent to $t_{P_t}$.
\end{proof}

If $q|q'$, then we can embedd $\mcR(F_m, \F_q)$ into $\mcR(F_m, \F_{q'})$,
and we can embedd $\mcR(F_m, \F_q)/\GammaL(2, q)$ into 
$\mcR(F_m, \F_{q'})/\GammaL(2, q')$ (where $\GammaL(2, q)$ acts on 
$\mcR(F_m, \F_q)$ by composition).

\begin{cor}
\label{C:bijection_max}
There is a bijection between $\MaxSpec(\Phi_m')$ and 
$\bigcup_q \mcR(F_m, \F_q)/\Gamma L(2, q)$, where $q$ ranges over all prime 
powers.
\end{cor}
\begin{proof}
This follows by Theorem~\ref{T:maxspec} and Corollary~\ref{C:alg_ind}.
\end{proof}

\subsection{Arbitrary fields}

\begin{definition}
Let $K$ and $K'$ be fields.
Let $t \in K^{\mcI_m}$ and $t' \in (K')^{\mcI_m}$ be trace tuples, and let $S$ 
and $S'$ be the rings generated by $t$ and $t'$, respectively. We say that $t$ 
and $t'$ are \textit{equivalent} if there exists a ring isomorphism 
$\alpha\colon S \to S'$ such that $\alpha(t_I) = t_I'$ for all $I \in \mcI_m$.
\end{definition}

\begin{remark}
By Corollary~\ref{C:alg_ind}, every trace tuple $t \in K^{\mcI_m}$ has a 
realization, but in general we must allow field extensions. That is, there 
exist matrices $A \in \SL(2, L)^m$ with $t_I = \tr(A_{i_1}\dotsb A_{i_k})$ for 
all $I = \{i_1 < \dotsb < i_k\} \in \mcI_m$, where $L$ is either $K$ or a 
quadratic extension of~$K$.
\end{remark}

Let $t \in K^{\mcI_m}$ be a trace tuple. Define a ring homomorphism 
$\alpha_t\colon \Phi_m' \to K$ by $\alpha_t(x_I) := t_I$ for $I \in \mcI_m$.
Then $P_t := \ker(\alpha_t)$ is a prime ideal of $\Phi_m'$.

Conversely, let $P  \in \Spec(\Phi_m')$, where $\Spec(\Phi_m')$ denotes the set 
of prime ideals of $\Phi_m'$. Let $K$ be the quotient field of $\Phi_m'/P$; set 
$(t_P)_I := x_I + P \in K$ for $I \in \mcI_m$.
Then $t_P := ((t_P)_I \mid I \in \mcI_m) \in K^{\mcI_m}$ is a trace tuple.

\begin{thm}
The maps $P \mapsto t_P$ and $t \mapsto P_t$ induce mutually inverse bijections 
between $\Spec(\Phi_m')$ and the set of equivalence classes of trace tuples.
\end{thm}

\section{Actions}
\label{S:actions}

\begin{definition}
Let $\Sigma_m := \{\pm 1\}^m$, the \textit{group of sign changes}.
Let $\Delta \in \mcR(F_m, K)$, and let 
$\chi\colon F_m \to \F_q\colon w \mapsto \tr(\Delta(w))$ be the character 
of~$\Delta$. Let $t \in K^{\mcI_m}$ be a trace tuple.

\begin{enumerate}
    \item
    Let $\sigma \in \Sigma_m$.
    Define 
    \begin{align*}
        {}^\sigma \Delta & \colon F_m \to \SL(2, K)\colon 
            w \mapsto w(\sigma)\Delta(w); \\
        {}^\sigma \chi & \colon F_m \to K\colon 
            w \mapsto w(\sigma)\chi(w); \text{ and}\\
        {}^\sigma t_I & := \Big(\prod_{i \in I}\sigma_i\Big)t_I.
    \end{align*}
    This defines actions of $\Sigma_m$ on representations, characters, and 
    trace tuples.
    
    \item
    Let $\sigma \in \Sigma_m$.
    Define a ring automorphism on $\Phi_m'$ by mapping $x_I$ to 
    $(\prod_{i \in I}\sigma_i)x_I$.
    This defines an action of $\Sigma_m$ on $\Phi_m'$ by automorphisms, and 
    hence an action on the set of ideals of $\Phi_m'$.

    \item
    Let $\alpha \in \Gal(K)$.
    Define 
    \begin{align*}
        {}^\alpha \Delta & \colon F_m \to \SL(2, K)\colon 
            w \mapsto \alpha(\Delta(w)); \\
        {}^\alpha \chi & \colon F_m \to K\colon 
            w \mapsto \alpha(\chi(w)); \text{ and} \\
        {}^\alpha t_I & := \alpha(t_I).
    \end{align*}
    This defines actions of $\Gal(K)$ on representations, characters, and trace 
    tuples.
\end{enumerate}
\end{definition}

\begin{remark}
The actions are compatible with the various bijections. More precisely, let 
$\Delta \in \mcR(F_m, K)$, let $t \in K^{\mcI_m}$ be a trace tuple, and let 
$P \in \Spec(\Phi_m')$. Denote by $\chi_\Delta$ the character of $\Delta$ and 
by $t_\Delta$ the trace tuple of $\Delta$. Then
\[
    \chi_{({}^\sigma\Delta)} = {}^\sigma(\chi_\Delta),\quad
    t_{({}^\sigma\Delta)} = {}^\sigma(t_\Delta),\quad
    P_{({}^\sigma t)} = {}^\sigma(P_t), \quad\text{and}\quad
    t_{({}^\sigma P)} = {}^\sigma(t_P)
\]
for all $\sigma \in \Sigma_m$, and
\[
    \chi_{({}^\alpha \Delta)} = {}^\alpha(\chi_\Delta) \quad\text{and}\quad 
        t_{({}^\alpha\Delta)} = {}^\alpha(t_\Delta)
\]
for all $\alpha \in \Gal(K)$.
\end{remark}

\section{Projective representations and finitely presented groups}
\label{S:projective}

\begin{definition}
Let $\Gamma$ be a group generated by $\gamma_1, \dotsc, \gamma_m$.
Set
\[
    \mcP(\Gamma, K) := \{\delta\colon \Gamma \to \PSL(2,K) \mid 
        \delta_{|\langle \gamma_1, \gamma_2 \rangle} 
        \text{ is absolutely irreducible}\}.
\]
\end{definition}

\begin{thm}
There is a bijection between $\MaxSpec(\Phi_m')/\Sigma_m$ and 
$\bigcup_q \mcP(F_m, \F_q)/\PGammaL(2, q)$, where $q$ ranges over all prime 
powers.
\end{thm}
\begin{proof}
This follows from Corollary~\ref{C:bijection_max}, since two representations 
$\Delta, \Delta'\colon F_m \to \SL(2, q)$ induce the same projective 
representation if and only if $\Delta' = {}^\sigma\Delta$ for some 
$\sigma \in \Sigma_m$.
\end{proof}

\begin{definition}
Let $G = \langle g_1, \dotsc, g_m \mid w_1, \dotsc, w_r\rangle$ be a finitely 
presented group. For $s \in \{\pm 1\}^r$ define
\[
    \I_s(G) := \langle \tau(w_ib) - s_i\tau(b)  \mid 
        1 \leq i \leq r,\ b \in \{1, g_1, g_2, g_1g_2\}\rangle \tri \Phi_m',
\]
the \textit{trace presentation ideal of $G$} with respect to the 
\textit{sign system} $s$. (We regard the $\tau(w)$ as elements of~$\Phi_m'$ via 
the isomorphism of Corollary~\ref{C:iso}.)
Set $\I(G) := \bigcap_{s \in \{\pm 1\}^r} \I_s(G)$, the 
\textit{full trace presentation ideal of~$G$}.
\end{definition}

The following result is a reformulation of \cite[Proposition~3.3]{l2q}.
\begin{prop}
\label{P:tracepresentation}
Let $G$ be a finitely presented group. Let $\Delta \in \mcR(F_m, \SL(2, K))$ 
with trace tuple $t \in K^{\mcI_m}$ and prime ideal 
$P = P_t \in \Spec(\Phi_m')$. The following are equivalent:
\begin{enumerate}
    \item
    The representation $\Delta$ induces a projective presentation 
    $\delta\colon G \to \PSL(2, K)$.

    \item
    The trace tuple $t$ is a zero of\/ $\I(G)$.

    \item
    The prime ideal $P$ contains\/ $\I(G)$.
\end{enumerate}
\end{prop}
\begin{proof}
The equivalence of (2) and (3) is immediate. We prove the equivalence of~(1) 
and~(2). Let $A_i := \Delta(g_i)$. Then $\Delta$ induces a projective 
representation of $G$ if and only if $w_i(A_1, \dotsc, A_m) = s_i I_2$ for some 
$s = (s_1, \dotsc, s_r) \in \{\pm 1\}^r$.
Since the trace bilinear form is non-degenerate, this is equivalent to 
$\tr(w_i(A_1, \dotsc, A_m)B) - s_i\tr(B) = 0$, where $B$ runs through a basis 
of~$K^{2 \times 2}$. Since $\langle A_1, A_2 \rangle$ is absolutely 
irreducible, we can choose the basis $(I_2, A_1, A_2, A_1A_2)$.
\end{proof}

\begin{cor}
\label{C:quotients}
There is a bijection between the maximal elements of $\V(\I(G))/\Sigma_m$,
where $\V(\I(G)) = \{P \in \Spec(\Phi_m') \mid \I(G) \subseteq P\}$ and 
$\bigcup_q \mcP(G, \F_q)/\PGammaL(2, q)$, where $q$ ranges over all prime 
powers.
\end{cor}

\section{Subgroups}
\label{S:subgrops}

Corollary~\ref{C:quotients} describes a bijection between classes of maximal 
ideals and classes of absolutely irreducible projective representations.
In this section, we establish criteria to decide whether a maximal ideal is 
mapped to a \textit{surjective} projective representation.

According to Dickson's classification (see for 
example~\cite[Section~3.6]{suzuki}), an absolutely irreducible subgroup 
$U \lneqq \PSL(2, q)$ is
\begin{itemize}
    \item
    isomorphic to $\Alt_4$, $\Sym_4$, or $\Alt_5$, or

    \item
    a dihedral group, or

    \item
    isomorphic to $\PGL(2, q')$ for some $q'|r$ if $q = r^2$ is a square, or

    \item
    isomorphic to $\PSL(2, q')$ for some $q'|q$.
\end{itemize}

For a finite group $H$ let $J(H) := \bigcap_{G} \I(G)$, where $G$ ranges over 
all presentations of $G$ on $m$ generators.

\begin{prop}
\label{P:subgroupideals}
Let $H$ be a finite group. 
Set $J'(H) := (J(H):\left(\bigcap_Q J(Q)\right)^\infty) \tri \Phi_m'$,
where $Q$ ranges over all proper quotients of $H$.

Let $\Delta \in \mcR(F_m, \SL(2, K))$ with trace tuple $t \in K^{\mcI_m}$ and 
prime ideal $P = P_t \in \Spec(\Phi_m')$. 
The following are equivalent:
\begin{enumerate}
    \item
    The representation $\Delta$ induces a projective presentation $\delta$
    such that $\im(\delta)\cong H$.

    \item
    The trace tuple $t$ is a zero of $J'(H)$.

    \item
    The prime ideal $P$ contains $J'(H)$.
\end{enumerate}
\end{prop}
\begin{proof}
It suffices to prove the equivalence of (1) and (2).
By Proposition~\ref{P:tracepresentation}, $\delta$ factors over $H$ if and only 
if $t$ is a zero of $J(H)$, and it factors over $Q$ if and only if $t$ is a 
zero of $J(Q)$. But $t$ is a zero of $J'(H)$ if and only if it is a zero of 
$J(H)$ but not a zero of $J(Q)$ for any proper quotient $Q$ of $H$, which 
proves the proposition.
\end{proof}

We will later let $H$ be one of the groups $\Alt_4$, $\Sym_4$, or $\Alt_5$,
which deals with the first kind of subgroups. We handle the dihedral groups in 
a slightly more general context.

\begin{lemma}
\label{L:allt}
Let $t \in K^{\mcI_m}$ be a trace tuple.
Let $\emptyset \neq J \subseteq \{1, \dotsc, m\}$.
If $t_I = 0$ for all $I \in \mcI_m$ with $|I \cap J|$ odd, then $t_I = 0$ for 
all $\emptyset \neq I \subseteq \{1, \dotsc, m\}$ with $|I \cap J|$ odd.
\end{lemma}
\begin{proof}
Assume $I \not\in \mcI_m$ with $|I \cap J|$ odd.
We proceed by induction on $|I|$.
We assume that $I \cap \{1, 2\} = \emptyset$; the other cases are analogous.
Let $i$ be the minimum of $I$, and let $j := I - \{i\}$.
By Proposition~\ref{P:linear_relations}, 
$t_I = t_{ij} = 1/\rho(t)\big(\lambda_0^i(t)t_j + \lambda_1^i(t)t_{1j} 
    + \lambda_2^i(t)t_{2j} + \lambda_{12}^i(t)t_{12j}\big)$.
There are eight cases to consider; we give the proof for two of them, the other 
six are analogous.
The first case is $1,2,i \not \in J$; the sets $j$, $\{1\}\cup j$, 
$\{2\}\cup j$, and $\{1,2\}\cup j$ have odd intersection with $J$, thus 
$t_j = t_{1j} = t_{2j} = t_{12j} = 0$ by induction. The formula for $t_{ij}$ 
shows that $t_{ij} = 0$.
The second case is $1 \in J$ but $2,i \not\in J$; now 
$t_1 = t_2 = t_i = t_{12} = t_{1i} = t_{2i} = t_{12i} = t_j = t_{2j} = t_{12j} 
    = 0$. 
By Proposition~\ref{P:t123}, $\lambda_1^i(t) = 0$, so $t_{ij} = 0$.
\end{proof}

Let $\Delta\in \mcR(F_m, K)$; then $\Delta$ is \textit{imprimitive} if 
$K^{2 \times 1} = V_1 \oplus V_2$ for one-dimensional subspaces 
$V_1, V_2 \leq K^{2 \times 1}$ such that $\Delta$ permutes the $V_i$ 
transitively.

\begin{prop}
\label{P:dihedral}
Let $K$ be an algebraically closed field.
Let $\Delta \in \mcR(F_m, K)$, and let $t$ be its trace tuple.
Then $\Delta$ is imprimitive if and only if there exists 
$\emptyset \neq J \subseteq \{1, \dotsc, m\}$ such that $t_I = 0$ for all 
$I \in \mcI_m$ with $|I \cap J|$ odd.
\end{prop}
\begin{proof}
Let $\chi\colon F_m \to \F_q\colon w \mapsto \tr(\Delta(w))$ be the character 
of $\Delta$. By \cite[Theorem~3.3]{character}, $\Delta$ is imprimitive if and 
only if there exists an epimorphism $\psi\colon F_m \to \{\pm 1\}$ such that 
$\psi(w) = -1$ implies $\chi(w) = 0$ for all $w \in F_m$.
For $\emptyset \neq J \subseteq \{1, \dotsc, m\}$ define an epimorphism 
$\psi_J\colon F_m \to \{\pm 1\}$ by $\psi_J(g_j) = -1$ if $j \in J$ and 
$\psi_J(g_j) = 1$ otherwise. This yields a bijection between the non-empty 
subsets of $\{1, \dotsc, m\}$ and the epimorphisms of $F_m$ onto $\{\pm 1\}$.
Let $A_i := \Delta(g_i)$ for $i \in \{1, \dotsc, m\}$. We show that 
$\psi_J(w) = -1$ implies $\chi(w)$ for all $w \in F_m$ if and only if $t_I = 0$
for all $I \in \mcI_m$ with $|I \cap J|$ odd.

The condition is obviously necessary; we show that it is sufficient.
By Lemma~\ref{L:allt} we may assume that $t_I = 0$ for all 
$\emptyset \neq I \subseteq \{1, \dotsc, m\}$ with $|I \cap J|$ odd.
Let $w \in F_m$ with $\psi_J(w) = -1$. We prove $\chi(w) = 0$ by induction on 
$|w|$, proceeding along the lines of the proof of Theorem~\ref{T:tau}. Note 
that $\chi(w) = \veps_A(\tau(w))$, where 
$A = (\Delta(g_1), \dotsc, \Delta(g_m))$.
If $w$ is conjugate to $g_i^{-1}w'$ for some $i \in \{1, \dotsc, m\}$ and some 
$w' \in F_m$ with $|w'| = |w| - 1$, then 
$\chi(w) = \chi(g_iw') - \chi(g_i)\chi(w')$.
By induction, $\chi(w) = \chi(g_iw')$, since either $\psi_J(g_iw')) = -1$ or 
$\psi_J(g_i) = -1$. Similar considerations apply to the other cases of the 
proof of Theorem~\ref{T:tau}, so we conclude $\chi(w) = 0$.
\end{proof}

The definition of imprimitivity depends on the field of definition.
By abuse of notation we call a representation imprimitive if it is imprimitive 
after field extension.

\begin{cor}
Let 
\[
    \mfD := \bigcap_{\emptyset \neq J \subseteq \{1, \dotsc, m\}}
        \langle x_I \mid 
            I \in \mcI_m \text{ with } |I \cap J| \text{ odd\/} \rangle 
            \tri \Phi_m'.
\]
Let $P \in \Spec(\Phi_m')$, and let $\Delta \in \mcR(F_m, K)$ be a realization 
of~$t_P$, where $K$ is the quotient field of $\Phi_m'/P$.
Then $\Delta$ is imprimitive if and only if $\mfD \subseteq P$.

In other words, the imprimitive representations correspond to the elements of 
the closed subset 
\[
    \V(\mfD) = \{P \in \Spec(\Phi_m') \mid \mfD \subseteq P\}
\]
of $\Spec(\Phi_m')$.
\end{cor}

The dihedral subgroups of $\PSL(2,q)$ are precisely the images of imprimitive 
subgroups of $\SL(2,q)$.
Setting $\mfA_4 := J'(\Alt_4)$, $\mfS_4 := J'(\Sym_4)$, and 
$\mfA_5 := J'(\Alt_5)$, we can formulate the main result of this section.

\begin{thm}
\label{T:quotientbijection}
Let $G$ be a finitely presented group on $m$ generators.
The set of normal subgroups $N \tri G$ such that $G/N \cong \PSL(2, q)$ for 
some prime power $q > 5$ or $G/N \cong \PGL(2, q)$ for some prime power $q > 4$ 
and such that $\langle g_1N, g_2N \rangle$ is absolutely irreducible is in 
bijection to the set of $\Sigma_m$-orbits of maximal ideals of
\[
    Q(G) := \V(\I(G)) - \V(\mfD \cap \mfA_4 \cap \mfS_4 \cap \mfA_5) 
        \subseteq \Spec(\Phi_m').
\]
\end{thm}

\section{The \texorpdfstring{$\PSL$-$\PGL$}{PSL-PGL}-decision}
\label{S:pglpsl}

\begin{definition}
A finite group is of \textit{$\Ln_2$-type} if it is isomorphic to $\PSL(2, q)$ 
for some $q > 5$ or to $\PGL(2, q)$ for some $q > 4$.
A quotient of a finitely presented group is an \textit{$\Ln_2$-quotient} if it 
is of $\Ln_2$-type.
\end{definition}

Theorem~\ref{T:quotientbijection} gives a characterization of $\Ln_2$-quotients
purely in algebro-geometric terms. To decide whether an $\Ln_2$-quotient is 
isomorphic to $\PSL(2,q)$ or $\PGL(2,q)$ for some~$q$, we use arithmetic tools.

Let $M \in Q(G)$ be a maximal ideal, and let $t_M$ be the trace tuple defined 
by~$M$. Let $\Delta\colon F_m \to \SL(2, q)$ be a realization of $t_M$.
The field $\Phi_m'/M$ is generated by $t_M$, so $\Phi_m'/M$ is the character 
field of $\Delta$. Since representations over finite fields can be realized 
over the character field, we may assume $q = |\Phi_m'/M|$.
If $q$ is not a square, then by Dickson's classification $\Delta$ induces an 
epimorphism onto $\PSL(2, q)$, and if $q = r^2$, then $\Delta$ induces an 
epimorphism onto $\PSL(2, q)$ or $\PGL(2, r)$.
We give a criterion to decide which case occurs.
Note that $\Delta$ induces a projective representation onto $\PGL(2, r)$ if and 
only if the image of $\Delta$ is conjugate to a subgroup of $\GL(2, r)\F_q^*$, 
where $\F_q^*$ is identified with scalar matrices.

\begin{prop}
\label{P:PGL}
Let $q = r^2$ be a prime power. Let $t \in \F_q^{\mcI_m}$ be a trace tuple and 
$\Delta\colon F_m \to \SL(2, q)$ a realization of~$t$, and let $\alpha$ be a 
generator of $\Gal(\F_q/\F_r)$.
The image of $\Delta$ is conjugate to a subgroup of $\GL(2, r)\F_q^*$ if and 
only if ${}^\sigma t = {}^\alpha t$ for some $\sigma \in \Sigma_m$.
\end{prop}
\begin{proof}
Let $\chi\colon F_m \to \F_q\colon w \mapsto \tr(\Delta(w))$ be the character 
of~$\Delta$. By \cite[Theorem~4.1]{character}, the image of $\Delta$ is 
conjugate to a subgroup of $\GL(2, r)\F_q^*$ if and only if there exists 
$\sigma \in \Sigma_m$ with ${}^\sigma\chi = {}^\alpha\chi$.
Using Lemma~\ref{L:allt} and the construction of the $\tau(w)$ in the proof of 
Theorem~\ref{T:tau}, we can show as in the proof of 
Proposition~\ref{P:dihedral} that this is equivalent to 
${}^\sigma t_I = {}^\alpha t_I$ for all $I \in \mcI_m$.
\end{proof}

\begin{remark}
Let $M \tri \Phi_m'$ be a maximal ideal, and let 
$\sigma \in \Stab_{\Sigma_m}(M)$. Then 
$\Phi_m'/M \to \Phi_m'/M\colon x_I + M \mapsto {}^\sigma x_I + M$
defines a Galois automorphism.
\end{remark}

\begin{cor}
Let $M \tri \Phi_m'$ be a maximal ideal such that $|\Phi_m'/M| = q = r^2$ is a 
square. Let $t = t_M$ be the trace tuple of $M$, and let 
$\Delta\colon F_m \to \SL(2, q)$ be a realization of $t_M$.
The image of $\Delta$ is conjugate to a subgroup of $\GL(2, r)\F_q^*$ if and 
only if $M$ has a non-trivial stabilizer in $\Sigma_m$.
\end{cor}

Together with Theorem~\ref{T:quotientbijection} we get the following result.

\begin{thm}
\label{T:quotientbijectionfull}
Let $G$ be a finitely presented group on $m$ generators.
The set of normal subgroups $N \tri G$ such that $G/N \cong \PSL(2, q)$ for 
some odd $q > 5$ with $\langle g_1 N, g_2 N\rangle$ absolutely irreducible is 
in bijection to the regular $\Sigma_m$-orbits of maximal ideals of~$Q(G)$.
The set of normal subgroups $N \tri G$ such that $G/N \cong \PGL(2, q)$ for 
some $q > 4$ with $\langle g_1 N, g_2 N\rangle$ absolutely irreducible is in 
bijection to the $\Sigma_m$-orbits of maximal ideals of $Q(G)$ with a 
stabilizer of order~$2$.
\end{thm}

When dealing with infinitely many $\Ln_2$-quotients, the following 
reformulation in terms of trace tuples if often useful.

\begin{cor}
\label{C:bijectiontuples}
Let $G$ be a finitely presented group on $m$ generators, and let $q = p^d$ be a 
prime power. If $q > 5$ is odd, then the set of normal subgroups $N \tri G$ 
such that $G/N \cong \PSL(2,q)$ with $\langle g_1N, g_2N \rangle$ absolutely 
irreducible is in bijection to the regular $\Sigma_m \times \Gal(\F_q)$-orbits 
of zeroes $t \in \F_q^{\mcI_m}$ of $Q(G)$ with $\F_q = \F_p[t]$.
If $q > 4$, then the set of normal subgroups $N \tri G$ such that 
$G/N \cong \PGL(2,q)$ with $\langle g_1N, g_2N \rangle$ absolutely irreducible 
is in bijection to the $\Sigma_m \times \Gal(\F_{q^2})$-orbits of zeroes 
$t \in \F_{q^2}^{\mcI_m}$ of $Q(G)$ with $\F_{q^2} = \F_p[t]$ having stabilizer 
of order~$2$.
\end{cor}

Let $G/N_1$ and $G/N_2$ be $\Ln_2$-quotients of $G$ with $N_1 \neq N_2$.
What is the isomorphism type of $G/N_1 \cap N_2$?
Clearly, if $G/N_1$ or $G/N_2$ is simple, then 
$G/N_1 \cap N_2 \cong G/N_1 \times G/N_2$.
This leaves the case that both $G/N_i$ are non-simple, that is, 
$G/N_i \cong \PGL(2, q_i)$ for some prime powers~$q_i$.

\begin{prop}
Let $G$ be a finitely presented group on $m$ generators.
Let $M_1$ and $M_2$ be maximal ideals of $Q(G)$ with stabilizers 
$\langle \sigma^{(i)} \rangle \leq \Sigma_m$ of order~$2$, and let 
$N_1, N_2 \tri G$ be normal subgrops corresponding to $M_1, M_2$ in the 
bijection of Theorem~\ref{T:quotientbijectionfull}.
Let $q_1, q_2$ be prime powers with $G/N_i \cong \PGL(2,q_i)$.
If $N_1 \neq N_2$, then
\[
    G/N_1 \cap N_2 \cong \begin{cases}
        \PGL(2,q_1) \Yup^{\Cyc_2} \PGL(2,q_2) & 
            \text{ if } \sigma^{(1)} = \sigma^{(2)},\\
        \PGL(2,q_1) \times \PGL(2,q_2) & \text{ otherwise.}
    \end{cases}
\]
\end{prop}
\begin{proof}
Let $\delta_i\colon G \to \PSL(2, q_i^2)$ be a realization of $M_i$; define 
$\delta_1 \times \delta_2\colon G \to \PSL(2, q_1^2) \times \PSL(2, q_2^2)
    \colon g \mapsto (\delta_1(g), \delta_2(g))$.
The image $H$ of $\delta_1 \times \delta_2$ is a subdirect product of 
$\PGL(2, q_1) \times \PGL(2, q_2)$. Since $N_1 \neq N_2$, this subdirect 
product is amalgamated either in $\Cyc_2$ or in the trivial group, and in the 
latter case the product is direct.
There is a unique epimorphism $\veps_i\colon \PGL(2, q_i) \to \Cyc_2$, where 
$\veps_i(\delta(g_j)) = 1$ if and only if $\delta_i(g_j) \in \PSL(2, q_i)$.
By the proof of \cite[Theorem~4.1]{character}, this is equivalent to 
$\sigma_j^{(i)} = 1$. Hence $\veps_1(\delta_1(g_j)) = \veps_2(\delta_2(g_j))$ 
if and only if $\sigma^{(1)} = \sigma^{(2)}$, which proves the proposition.
\end{proof}

\section{Arbitrary representations}
\label{S:arbitrary}

Until now, we only considered representations $\Delta\colon F_m \to \SL(2, K)$ 
such that $\Delta_{|\langle g_1, g_2 \rangle}$ is absolutely irreducible. We 
now show how the case of arbitrary absolutely irreducible representations can 
be reduced to this one.

\begin{prop}
\label{P:anna}
Let $\Delta\colon F_m \to \SL(2, K)$ be a representation.
For $1 \leq i,j,k \leq m$ set 
$\Delta_{i,j} := \Delta_{|\langle g_i, g_j \rangle}$ and 
$\Delta_{i,jk} := \Delta_{|\langle g_i, g_jg_k \rangle}$.
Then $\Delta$ is absolutely irreducible if and only if one of $\Delta_{i,j}$ 
with $1 \leq i < j \leq m$, $\Delta_{1,2i}$ with $3 \leq i \leq m$, or 
$\Delta_{2,ij}$ with $3 \leq i < j \leq m$ is absolutely irreducible.
\end{prop}
\begin{proof}
We generalize \cite[Lemma~3.4.4]{fabianska} and so strenghten 
\cite[Proposition~B.7]{brumfiel}. We may assume that $K$ is algebraically 
closed, so absolute irreducibility coincides with irreducibility.
Clearly if some restriction of $\Delta$ is irreducible, then $\Delta$ is 
irreducible. So assume now that all given restrictions are reducible. We show 
that $\Delta$ is reducible. Since $\Delta_{i,j}$ is reducible, $\Delta(g_i)$ 
and $\Delta(g_j)$ have a common eigenspace. If the minimal polynomial of some 
$\Delta(g_i)$ is not square-free, then $\Delta(g_i)$ has a unique eigenspace of 
dimension~$1$, which has to be a common eigenspace for all $\Delta(g_j)$.
Thus $\Delta$ is reducible. So assume now that the minimal polynomials of all 
$\Delta(g_i)$ are square-free. We may further assume that all $\Delta(g_i)$ 
have two distinct eigenvalues; for if $\Delta(g_i)$ is a scalar matrix, then 
$\Delta$ is reducible if and only if 
$\Delta_{|\langle g_1, \dotsc, \wh{g_i}, \dotsc, g_m \rangle}$ is reducible.
Let $E_i$ be the set of eigenspaces of $\Delta_i$ and 
$\mcE := \{E_i  \mid 1 \leq i \leq m\}$. By our hypothesis, 
$|E_i \cap E_j| \geq 1$ for all $i,j$. Note that $|E_i| = 2$, so if 
$|\mcE| \geq 4$, then the $E_i$ must have a common element, that is, the 
matrices have a common eigenspace. 
The same is trivially true if $|\mcE| \leq 2$. 
Assume now that $|\mcE| = 3$. Consider first the case $E_1 \neq E_2$.
Let $E_1 = \{\langle v_1 \rangle, \langle v_2 \rangle\}$ and 
$E_1 \cap E_2 = \{\langle v_1 \rangle\}$.
We claim that $\langle v_1 \rangle$ is a common eigenspace for all 
$\Delta(g_i)$. For suppose that $\langle v_1 \rangle$ is not an eigenspace of 
$\Delta(g_i)$ for some~$i$; then $\langle v_2 \rangle$ must be an eigenspace 
of~$\Delta(g_i)$, since $|E_1 \cap E_i| \geq 1$. Since $\Delta_{1,2i}$ is 
reducible, $\Delta(g_1)$ and $\Delta(g_2g_i)$ have a common eigenspace.
This is either $\langle v_1 \rangle$ or $\langle v_2 \rangle$.
In the first case, $\Delta(g_2g_i)$ and $\Delta(g_2)$ have eigenspace 
$\langle v_1 \rangle$, so $\Delta(g_i)$ has eigenspace $\langle v_1 \rangle$, 
contradicting our assumption.
In the second case, $\Delta(g_2g_i)$ and $\Delta(g_i)$ have eigenspace 
$\langle v_2 \rangle$, so $\Delta(g_2)$ has eigenspace $\langle v_2 \rangle$, 
again a contradiction.
Thus the assumption that $\langle v_1 \rangle$ is not an eigenspace of 
$\Delta(g_i)$ is impossible.
We conclude the proof by showing that $E_1 = E_2$ is not possible.
Since $|\mcE| = 3$, there exist $i < j$ with $\mcE = \{E_1, E_i, E_j \}$. 
All sets have at least one element in common, so we may assume 
$E_1 = \{\langle v_1 \rangle, \langle v_2 \rangle \}$,
$E_i = \{ \langle v_1 \rangle, \langle v_3 \rangle\}$, and 
$E_j = \{\langle v_2 \rangle, \langle v_3 \rangle\}$.
Since $\Delta_{2,ij}$ is reducible, $\Delta(g_2)$ and $\Delta(g_ig_j)$ have a 
common eigenspace. Assume that this is $\langle v_1 \rangle$; then 
$\langle v_1 \rangle$ is also an eigenspace of $\Delta(g_j)$, a contradiction. 
If it is $\langle v_2 \rangle$, then $\langle v_2 \rangle$ is also an 
eigenspace of $\Delta(g_i)$, also a contradiction. Thus $E_1 = E_2$ is 
impossible.
\end{proof}

Let 
\[
    U_m := \{(\{i\}, \{j\}) \mid 1 \leq i < j \leq m\} \cup 
        \{(\{1\}, \{2,j\}) \mid 3 \leq j \leq m\} \cup 
        \{(\{2\}, \{i,j\}) \mid 3 \leq i < j \leq m\}.
\]
For every $u = (u_1, u_2) \in U_m$, let $\alpha_u \in \Aut(F_m)$ with 
$\alpha_u(g_{u_1}) = g_1$ and $\alpha_u(g_{u_2}) = u_2$, where 
$g_v := g_{v_1}\dotsb g_{v_k}$ for $v = \{v_1 < \dotsb < v_k\}$.
Thus $\Delta\colon F_m \to \SL(2,K)$ is absolutely irreducible if and only if 
$(\Delta\circ \alpha_u^{-1})_{|\langle g_1, g_2 \rangle}$ is absolutely 
irreducible for some $u \in U_m$.

By abuse of notation, if $\alpha \in \Aut(F_m)$ and $G$ is a group generated by 
elements $g_1, \dotsc, g_m$, then we denote the automorphism of $G$ defined by 
$g_i \mapsto \alpha(g_i)$ for $1 \leq i \leq m$ again by $\alpha$.
Fix a total order $<$ on~$U_m$. Set
\[
    \I_u(G) := \I(\alpha(G)) + 
        \langle \rho(x_{v_1}, x_{v_2}, x_{v_1 \cup v_2}) \mid 
            v \in U_m, v < u\rangle.
\]

For a maximal ideal $M \in \V(\I_u(G))$ let $t_M$ be the trace tuple, and let 
$\Delta_M \colon F_m \to \SL(2, q)$ be a realization of~$t_M$, where 
$q = |\Phi_m'/M|$. The projective representation induced by 
$\Delta_M \circ \alpha_u$ factors over $G$; denote this projective 
representation by $\delta_{M, u}$, and define 
$N_{M} := \ker(\delta_{M, u}) \tri G$.
Note that $N_{M}$ is constant on the $\Sigma_m$-orbit of~$M$.

Conversely, let $N \tri G$ such that $G/N$ is of $\Ln_2$-type.
Let $\delta\colon G \to \PSL(2, q)$ with $\ker(\delta) = N$, and let 
$\Delta\colon F_m \to \SL(2,q)$ be a lift of~$\delta$. Set $t = t_\Delta$ and 
$M_N := P_t$; then $M_N$ is a maximal $\Ln_2$-ideal. Note that $M_N$ is only 
well-defined up to the action of~$\Sigma_m$.
If $u \in U_m$ is minimal such that 
$(\Delta\circ \alpha_{u}^{-1})|_{|\langle g_1, g_2\rangle}$ is absolutely 
irreducible, then $M_N \in \V(\I_{u}(G))$.

For $u \in U_m$ set
\[
    Q_u(G) := \V(\I_u(G)) - V(\mfD \cap \mfA_4 \cap \mfS_4 \cap \mfA_5).
\]

We now present the main result.

\begin{thm}
Let $G$ be a finitely presented group on $m$ generators.
The maps $M \mapsto N_{M}$ and $N \mapsto M_N$ induce mutually inverse 
bijections between $\Sigma_m$-orbits of maximal ideals of 
$\biguplus_{u \in U_m}Q_u(G)$ and normal subgroups $N \tri G$ such that $G/N$ 
is of $\Ln_2$-type (where $\biguplus$ denotes the disjoint union).
\end{thm}
\begin{proof}
This follows by Proposition~\ref{P:anna} and 
Theorem~\ref{T:quotientbijectionfull}.
\end{proof}

\section{Subgroup tests}
\label{S:subgrouptests}

Proposition~\ref{P:subgroupideals} allows us to test whether a realization 
$\Delta\colon F_m \to \SL(2, K)$ of a prime ideal $P \tri \Phi_m'$ maps 
projectively onto $\Alt_4$, $\Sym_4$, or~$\Alt_5$, using the ideals 
$J'(\Alt_4)$, $J'(\Sym_4)$, and~$J'(\Alt_5)$.
These ideals are easily computed if $m = 2$, since there are only $4$ 
presentations of $\Alt_4$ on two generators, $9$ for~$\Sym_4$, and $19$ 
for~$\Alt_5$; see \cite[Lemmas~3.7--3.9]{l2q}.
However, this approach is no longer efficient if $m \geq 3$.
For example, there are $65$ presentations of $\Alt_4$ on three generators,
$420$ for~$\Sym_4$, and $1688$ for~$\Alt_5$.

In this section, we describe a more efficient test, using the absolutely 
irreducible subgroups of $\Alt_4$, $\Sym_4$, and $\Alt_5$. Set 
$A_i := \Delta(g_i)$ and let $a_i \in \PSL(2, K)$ be the projective image, for 
$1 \leq i \leq m$.
We assume that $\langle A_1, A_2 \rangle$ is absolutely irreducible.
Define $H := \langle a_1, \dotsc, a_m \rangle$.
If $H \cong \Alt_4$, then $\langle a_1, a_2 \rangle \in \{\V_4, \Alt_4\}$; 
if $H \cong \Sym_4$, then 
$\langle a_1, a_2 \rangle \in \{\V_4, \Sym_3, \Die_8, \Alt_4, \Sym_4\}$; 
and if $H \cong \Sym_4$, then 
$\langle a_1, a_2 \rangle \in \{\V_4, \Sym_3, \Die_{10}, \Alt_4, \Sym_4\}$.
It is easy to check whether 
$\langle a_1, a_2 \rangle  \in 
    \{\V_4, \Sym_3, \Die_8, \Die_{10}, \Alt_4, \Sym_4, \Alt_5\}$;
for example, $\langle a_1, a_2 \rangle = \V_4$ if and only if 
$\tr(A_1) = \tr(A_2) = \tr(A_1A_2) = 0$.
If $\langle a_1, a_2 \rangle$ is one of the seven groups, then we can always 
find matrices $B_1 = w_1(A_1, A_2), B_2 = w_2(A_1, A_2)$ such that 
$\tr(B_1) = \tr(B_2) = 0$ and $\langle w_1(a_1, a_2), w_2(a_1, a_2) \rangle$ is 
a dihedral group of order~$4$, $6$, or~$10$. In the latter two cases we may 
also assume that $\tr(B_1B_2) = 1$ or $\tr(B_1B_2)$ is a root of $X^2 + X - 1$,
respectively.

For $B = (B_1, B_2) \in \SL(2,q)^2$ and $X \in \SL(2,q)$ let 
\[
    \theta_B(X) := (\tr(X), \tr(B_1X), \tr(B_2X), \tr(B_1B_2X)) \in \F_q^4.
\]
If $\langle B_1, B_2 \rangle$ is absolutely irreducible, then $X$ is uniquely 
determined by $\theta_B(X)$, see~Proposition~\ref{P:t123}.

We give details of an $\Alt_4$-test.
Fix $B$, and let $b_i \in \PSL(2, q)$ be the projective image of~$B_i$.
Assume $\langle b_1, b_2 \rangle \cong \V_4$; let 
$\langle b_1, b_2 \rangle \leq \Gamma \leq \PSL(2, q)$ with 
$\Gamma \cong \Alt_4$ and let $\wt{\Gamma} \leq \SL(2, q)$ be the full preimage 
of $\Gamma$. Now $X \in \SL(2, q)$ maps onto an element of $\Gamma$ if and only 
if $\theta_B(X) \in \theta_B(\wt{\Gamma}) 
    = \{\theta_B(Y) \mid Y \in \wt{\Gamma}\}$,
thus for an effective subgroup test it is enough to compute the sets 
$\theta_B(\wt{\Gamma})$.
The subgroups of $\PSL(2, q)$ isomorphic to $\Alt_4$ are all conjugate 
in~$\PGL(2, q)$, and
$\theta_{({}^MB)}(\wt{\Gamma}) = \theta_{B}({}^{M^{-1}}\wt{\Gamma})$ for all 
$M \in \GL(2, q)$, so it is enough to compute $\theta_B(\wt{\Gamma})$ for a 
fixed $\Gamma$ and all possible~$B$.
Furthermore, $\theta_{({}^MB)}(\wt{\Gamma}) = \theta_B(\wt{\Gamma})$ for all 
$M \in \Nor_{\GL(2, q)}(\wt{\Gamma})$, so it suffices to compute 
$\theta_B(\wt{\Gamma})$ for a fixed $\Gamma$ and all 
$\Nor_{\GL(2, q)}(\wt{\Gamma})$-conjugacy classes of pairs $B \in \wt{\Gamma}$ 
mapping onto generators for~$\V_4$.
Finally, the subgroups $\wt{\Gamma}$ are up to conjugation images of 
$\SL(2, 3) \leq \SL(2, \Z[i])$ modulo a prime ideal of $\Z[i]$, so 
$\theta_B(\wt{\Gamma})$ can be computed uniformly for all prime powers~$q$ by a 
single computation over~$\Z$.

Summarizing, we get the following result.

\begin{prop}
\label{P:tracesV4}
Let $G = \PSL(2, q)$ for an odd prime power~$q$, let $a_1, a_2 \in G$ be 
generators of a Klein four group $V$, and let $z \in G$.
Let $A_i \in \SL(2, q)$ be a preimage of $A_i$, and let $Z \in \SL(2, q)$ be a 
preimage of $z$.

There is a unique $H \leq G$ isomorphic to $\Alt_4$ which contains $V$,
and $z \in H$ if and only if $\theta_B(Z)$ is one of the $24$ elements
\[
    \Theta_4 := \{(\pm 2, 0, 0, 0), (0, \pm 2, 0, 0), (0, 0, \pm 2, 0), 
        (0, 0, 0, \pm 2), (\pm 1, \pm 1, \pm 1, \pm 1)\}.
\]
\end{prop}

\begin{prop}
Let $A_1, \dotsc, A_m \in \SL(2,q)$ such that $\langle A_1, A_2 \rangle$ is 
absolutely irreducible. Let $t = (\tr(A_1), \tr(A_2), \tr(A_1A_2))$, and let 
$a_i \in \PSL(2,q)$ be the image of $A_i$.
Set $B := (A_1, A_2)$.
Then $\langle a_1, \dotsc, a_m \rangle$ is isomorphic to $\Alt_4$ if and only 
if one of the following conditions is satisfied.
\begin{enumerate}
    \item
    $t = (0,0,0)$ and $\theta_B(A_i) \in \Theta_4$ for all $3 \leq i \leq m$, 
    where $B = (A_1, A_2)$, and at least one 
    $\theta_B(A_i) = (\pm 1, \pm 1, \pm 1, \pm 1)$.

    \item
    $t = (0,\pm 1,\pm 1)$ and $\theta_B(A_i) \in \Theta_4$ for all 
    $3 \leq i \leq m$, where $B = (A_1, A_2^{-1}A_1A_2)$.

    \item
    $t = (\pm 1, 0, \pm 1)$ and $\theta_B(A_i) \in \Theta_4$ for all 
    $3 \leq i \leq m$, where $B = (A_2, A_1^{-1}A_2A_1)$.

    \item
    $t = (\pm 1, \pm 1, 0)$ and $\theta_B(A_i) \in \Theta_4$ for all 
    $3 \leq i \leq m$, where $B = (A_1A_2, A_2A_1)$.

    \item
    $t = (\pm 1, \pm 1, \pm 1)$ with an even number of $-1$'s, and 
    $\theta_B(A_i) \in \Theta_4$ for all $3 \leq i \leq m$, where 
    $B = (A_1A_2^{-1}, A_2^{-1}A_1)$.
\end{enumerate}
\end{prop}
\begin{proof}
The only absolutely irreducible subgroups of $\Alt_4$ are the Klein four group 
and $\Alt_4$. If $t = (0,0,0)$, then $\langle a_1, a_2 \rangle$ is a Klein four 
group, and the claim follows by Proposition~\ref{P:tracesV4}.
If $t = (0,\pm 1, \pm 1)$, then $\langle a_1, a_2 \rangle = \Alt_4$, and 
$\langle a_1, a_2^{-1}a_1a_2 \rangle$ generate the subgroup of order~$4$; 
again, the claim follows by Proposition~\ref{P:tracesV4}.
The other cases correspond to the other three presentations of $\Alt_4$ and are 
handled similarly.
\end{proof}

It is straight-forward to give similar conditions for $\Sym_4$ and $\Alt_5$, 
utilizing the subgroups $\Sym_3$ and $\Die_{10}$ in addition to~$\V_4$.

\section{\texorpdfstring{$\Ln_2$}{L2}-ideals}
\label{S:l2ideals}

\begin{definition}
An \textit{$\Ln_2$-ideal} is a prime ideal 
$P \in \Spec(\Phi_m') - \V(\mfD \cap \mfA_4 \cap \mfS_4 \cap \mfA_5)$.
Let $P \cap \Z = \langle p \rangle$, and let $d$ be the Krull dimension of~$P$.

\begin{enumerate}
    \item
    If $d = 0$, that is, $P$ is a maximal ideal, let 
    $k := \dim_{\F_p}(\Phi_m'/P)$.
    Then $P$ is of type $\Ln_2(p^k)$ if $\Stab_{\Sigma_m}(P)$ is trivial,
    and of type $\PGL(2, p^{k/2})$ otherwise.

    \item
    If $d > 0$ and $p \neq 0$, then $P$ is of type $\Ln_2(p^{\infty^d})$, or of 
    type $\Ln_2(p^\infty)$ if $d = 1$.

    \item
    If $d = 1$ and $p = 0$, let 
    $k := \dim_{\Q}(\Phi_m'\otimes_{\Z}\Q/P\otimes_{\Z}\Q)$.
    Then $P$ is of type $\Ln_2(\infty^k)$.

    \item
    If $d > 1$ and $p = 0$, then $P$ is of type $\Ln_2(\infty^{\infty^{d-1}})$, 
    or of type $\Ln_2(\infty^\infty)$ if $d = 2$.
\end{enumerate}
\end{definition}

For an $\Ln_2$-ideal $P$ let $t_P$ be the trace tuple, $\Delta_P$ a realization 
of $t_P$, and $\delta_P$ the projective representation induced by $\Delta_P$.

\begin{prop}
Let $P$ be an $\Ln_2$-ideal.
\begin{enumerate}
    \item
    If $P$ is of type $\Ln_2(p^k)$, then the image of $\delta_P$ is isomorphic 
    to $\Ln_2(p^k)$; if $P$ is of type $\PGL(2,p^k)$, then the image of 
    $\delta_P$ is isomorphic to $\PGL(2,p^k)$.

    \item
    If $P$ is of type $\Ln_2(\infty^k)$, then 
    every maximal $\Ln_2$-ideal containing $P$ is of type $\Ln_2(p^\ell)$ or 
    $\PGL(2, p^{\ell/2})$ with $\ell \leq k$.
    Moreover, the set of maximal elements of $\V(P)$ which are not 
    $\Ln_2$-ideals is finite.

    \item
    If $P$ is of type $\Ln_2(p^{\infty^d})$, then there are infinitely many 
    $k \in \N$ such that $\V(P)$ contains $\Ln_2$-ideals of type $\Ln_2(p^k)$.
    Moreover, the set of prime ideals in $\V(P)$ which are not $\Ln_2$-ideals 
    form a closed set of dimension at most~$d-1$.

    \item
    If $P$ is of type $\Ln_2(\infty^{\infty^{d-1}})$, then for all but finitely 
    many primes $p$ there exist infinitely many $k \in \N$ such that $\V(P)$ 
    contains $\Ln_2$-ideals of type $\Ln_2(p^k)$.
    Moreover, the set of prime ideals in $\V(P)$ which are not $\Ln_2$-ideals 
    form a closed set of dimension at most~$d-1$.
\end{enumerate}
\end{prop}
\begin{proof}
First note that the set of prime ideals in $\V(P)$ which are not $\Ln_2$-ideals 
are precisely the elements of the set 
$\V(P) \cap \V(\mfD \cap \mfA_4 \cap \mfS_4 \cap \mfA_5) 
    = \V(P + \mfD \cap \mfA_4 \cap \mfS_4 \cap \mfA_5)$.
Since $P$ does not contain $\mfD \cap \mfA_4 \cap \mfS_4 \cap \mfA_5$ and $P$ 
is prime, $P \subsetneqq P + \mfD \cap \mfA_4 \cap \mfS_4 \cap \mfA_5$, so the 
Krull dimension of the latter ideal is smaller than that of~$P$. This settles 
all claims about the size of $\V(P)$.

We prove the other claims.
\begin{enumerate}
    \item
    This follows by Theorem~\ref{T:quotientbijectionfull}.

    \item
    The first point follows since 
    $\dim_{\F_p}(\Phi_m'\otimes_{\Z}\F_p/P \otimes_{\Z} \F_p) \leq 
        \dim_{\Q}(\Phi_m'\otimes_{\Z}\Q/P \otimes_{\Z} \Q)$
    for all primes $p$, and the maximal ideals of $\Phi_m'$ containing $P$ and 
    $p$ are in bijection to the maximal ideals of $\Phi_m'\otimes_{\Z}\F_p$ 
    containing $P \otimes_{\Z}\F_p$. For the second point, note that a set of 
    dimension zero is finite.

    \item
    Since $\Phi_m'$ is finitely generated, there are only finitely many 
    epimorphisms of $\Phi_m'$ onto $\F_{p^k}$ for every~$k$.
    But there are infinitely many primes containing $P$.

    \item
    In this case, $(\Phi_m'/P)\otimes_{\Z}\Q$ has algebraically independent 
    elements, so there are epimorphisms onto number fields of arbitrarily high 
    degrees.
    Let $\alpha \colon \Phi_m'\otimes_{\Z}\Q \to K$ be an epimorphism onto a 
    number field $K$ of degree~$k$ such that $\alpha$ factors over 
    $\Phi_m'\otimes_{\Z}\Q$; set $Q := \ker(\alpha_{|\Phi_m'}) \tri \Phi_m'$.
    Then $Q \supseteq P$; if $Q$ is an $\Ln_2$-ideal, then it is of type 
    $\Ln_2(\infty^k)$.
    But the prime ideals which are not $\Ln_2$-ideals form a set of 
    Krull-dimension $d-1$, so this approach yields an $\Ln_2$-ideal for almost 
    all~$Q$. The result now follows by part~(2).
    \qedhere
\end{enumerate}
\end{proof}

\section{The algorithm}
\label{S:algorithm}

\begin{definition}
Let 
$G = \langle g_1, \dotsc, g_m \mid 
    w_1(g_1, \dotsc, g_m), \dotsc, w_r(g_1, \dotsc, g_m) \rangle$ 
be a finitely presented group. Then $\Sigma_m$ acts on the set $\{\pm 1\}^r$ 
of sign systems by
\[
    {}^\sigma s := (w_1(\sigma_1, \dotsc, \sigma_m)s_1, \dotsc, 
        w_r(\sigma_1, \dotsc, \sigma_m)s_r)
\]
for $\sigma \in \Sigma_m$ and $s \in \{\pm 1\}^r$.
\end{definition}

\begin{remark}
A prime ideal $P \in \Spec(\Phi_m')$ contains $\I_s(G)$ if and only if 
${}^\sigma P$ contains $\I_{({}^\sigma s)}(G)$. Let $T$ be the kernel of the 
action and $S$ a set of representatives of the orbits; then the 
$\Sigma_m$-orbits of $\V(\I(G))$ are in bijection to the $T$-orbits of 
$\V(\bigcap_{s \in S} \I_s(G))$.
\end{remark}

This allows us to reduce the computations by a factor up to $2^m$.

\begin{algorithm}[\texttt{L2Quotients}]
\textit{Input:} A finitely presented group~$G$.

\noindent\textit{Output:} For every $u \in U_m$, a set of representatives for 
the $\Sigma_m$-orbits of minimal $\Ln_2$-ideals of $Q_u(G)$.

\begin{enumerate}
    \item
    Set $A := U_m$ and $\mcR := \emptyset$.

    \item
    \label{Step:loop}
    Let $u$ be the smallest element in $A$.
    Let $T$ be the kernel and $S$ a set of representatives for the orbits of 
    the action of $\Sigma_m$ on the sign systems of $\alpha_u(G)$.

    \item
    Let $\mcP$ be the set of minimal elements in 
    $\bigcup_{s \in S} \MinAss(\I_s^u(G))$, where $\MinAss(I)$ denotes the 
    minimal associated prime ideals of~$I$.
    Remove from $\mcP$ all elements which contain one of the ideals $\mfD$, 
    $\mfA_4$, $\mfS_4$, or $\mfA_5$.

    \item
    Choose a set $\mcP'$ of representatives of $T$-orbits on $\mcP$.

    \item
    Add $(\mcP', u)$ to $\mcR$, and remove $u$ from $A$.
    If $A \neq \emptyset$, go to step~\ref{Step:loop}; otherwise return $\mcR$.
\end{enumerate}
\end{algorithm}

\begin{remark}
\begin{enumerate}
    \item
    The output of the algorithm describes the $\Ln_2$-quotients of $G$ as 
    follows. For every $N \tri G$ with $G/N$ of $\Ln_2$-type there exists 
    $u \in U_m$ and $\sigma \in \Sigma_m$ such that $M_N \supseteq P$ for 
    some~$P$. Conversely, if $M \supseteq P$ is a maximal $\Ln_2$-ideal for 
    some~$P$, then $N_M \tri G$ with $G/N_M$ of $\Ln_2$-type.

    \item
    If all prime ideals returned by the algorithm are maximal, then $G$ has 
    only finitely many $\Ln_2$-quotients, and the normal subgroups $N \tri G$ 
    with $G/N$ of $\Ln_2$-type are in bijection to the maximal ideals.

    \item
    If the algorithm returns at least one prime ideal of positive Krull 
    dimension, then $G$ has infinitely many $\Ln_2$-quotients.
\end{enumerate}
\end{remark}

The algorithm has been implemented in \textsc{Magma} \cite{magma}.

\begin{remark}
\begin{enumerate}
    \item
    The ring $\Phi_m'$ is very useful for the theoretical description of the 
    algorithm. However, in practice the localization at $\rho$ slows down 
    computations considerably. Instead, we work with the preimage of 
    $\I_s^u(G)$ in $\Z[x_J \mid \emptyset \neq J \subseteq \{1, \dotsc, m\}]$, 
    and remove all prime components containing~$\rho$.

    \item
    The implementation uses Gr\"obner bases to handle the ideals $\I_s^u(G)$.
    However, Gr\"obner basis computations over the integers can be very slow,
    especially as $m$ grows. The algorithm in \cite{minass} to compute the 
    minimal associated primes of an ideal replaces Gr\"obner basis computations 
    over the integers by several Gr\"obner basis computations over prime 
    fields, resulting in a much faster algorithm.
\end{enumerate}
\end{remark}

\subsection{Adaptation to Coxeter groups}

Coxeter groups are a special class of finitely presented groups, where the only 
relations are $(g_ig_j)^{C_{ij}} = 1$ for a symmetric matrix 
$C = (C_{ij}) \in (\Z \cup \{\infty\})^{m \times m}$ with $1$'s along the 
diagonal (if $C_{ij} = \infty$, then we simply omit the relation).
We call $C$ a \textit{Coxeter matrix} and denote the finitely presented group 
by~$G_C$. We are often only interested in smooth quotients of Coxeter groups, 
that is, those for which the images also have the prescribed orders (unless the 
prescribed order is~$\infty$). In this case, the $\Ln_2$-quotient algorithm can 
be simplified, which also results in a considerable speed-up of the 
computation. This is based on the following.

For $n \in \N$ let $\zeta_n \in \C$ be a primitive $n$-th root of unity. 
Set $\eta_n := \zeta_n + \zeta_n^{-1}$, and let $\Psi_n \in \Z[T]$ be the 
minimal polynomial of~$\eta_n$.
For convenience, we define $\Psi_\infty := 0$.

\begin{remark}
\label{R:order}
Let $A \in \SL(2, K)$ where $k$ is a field of characteristic $p \geq 0$, and 
let $n \in \N$.

\begin{enumerate}
    \item
    If $p = 0$ or $(n,p) = 1$, then $\Psi_n(\tr(A)) = 0$ if and only if 
    $|A| = n$.

    \item
    If $n = p$, then $\Psi_n(\tr(A)) = 0$ if and only if $|A| \in \{1, p\}$.

    \item
    If $n = 2p \neq 4$, then $\Psi_n(\tr(A)) = 0$ if and only if 
    $|A| \in \{2, 2p\}$.
\end{enumerate}
\end{remark}

For a Coxeter matrix $C \in (\Z \cup \{\infty\})^{m \times m}$ set
\begin{multline*}
    \I(C) := \langle x_1, \dotsc, x_m \rangle + 
        \langle \Psi_{2C_{ij}}(x_{ij}) \mid 
            1 \leq i < j \leq m \text{ with $C_{ij}$ even} \rangle \\
    + \langle \Psi_{C_{ij}}(x_{ij})\Psi_{2C_{ij}}(x_{ij}) \mid 
        1 \leq i < j \leq m \text{ with $C_{ij}$ odd} \rangle \tri \Phi_m',
\end{multline*}
where 
$x_{ij} = \rho^{-1}(\lambda_{0}^{i}x_j + \lambda_{1}^{i}x_{1j} 
    + \lambda_{2}^{i}x_{2j} + \lambda_{12}^{i}x_{12j})$.

\begin{remark}
Let $a_1, a_2 \in \Ln_2(q)$ with $|a_1| = |a_2| = 2$ and $|a_1a_2| \neq 1$.
Then $\langle a_1, a_2 \rangle$ is absolutely irreducible if and only if 
$(q, |a_1a_2|) = 1$.
\end{remark}

\begin{thm}
Let $C \in (\Z \cup \{\infty\})^{m \times m}$ be a Coxeter matrix.
\begin{enumerate}
    \item
    Let $q = p^d$, and let $\Delta\colon F_m \to \SL(2, q)$ be a representation 
    which induces a smooth projective representation 
    $\delta\colon G_C \to \PSL(2,q)$ such that $\delta(G_C)$ is of 
    $\Ln_2$-type. Let $t := t_\Delta$ and $P := P_t$.
    If $|\delta(g_1g_2)| \neq p$, then $P \supseteq \I(C)$.

    \item
    Let $M \supseteq \I(C)$ be a maximal $\Ln_2$-ideal and 
    $\Delta = \Delta_M\colon F_m \to \SL(2, q)$ a realization.
    Then $\Delta$ induces a projective representation 
    $\delta\colon G_C \to \PSL(2, q)$ such that $\delta(G_C)$ is of 
    $\Ln_2$-type.
    If $(q,2C_{ij}) = 1$ for all $1 \leq i < j \leq m$, then $\delta$ is smooth.
\end{enumerate}
\end{thm}
\begin{proof}
This follows easily by the preceeding remarks.
\end{proof}

This can be easily turned into an algorithm. We leave the details to the 
reader.

\subsection{Computing realizations}

The $\Ln_2$-quotient algorithm returns a set of $\Ln_2$-ideals, which contain a 
lot of information, for example, the isomorphism types and number of 
$\Ln_2$-images. However, in certain cases one will want to compute an explicit 
epimorphism $G \to \PSL(2, q)$ encoded by an $\Ln_2$-ideal.
We now present an algorithm to accomplish that. This algorithm works for 
representations of arbitrary degree, so we present it in this generality.

\begin{prop}
Let $G$ be a finitely generated group, and let $\chi\colon G \to K$ be the 
character of an absolutely irreducible representation $\Delta$ of degree~$n$.
There is a probabilistic algorithm with input $\chi$ and $n$ which constructs 
an extension field $L/K$ of degree at most~$n$ and a representation 
$\Delta'\colon G \to \GL(n, L)$, such that $\Delta'$ is equivalent to $\Delta$.
If $K$ is finite, then we can choose $L = K$.
\end{prop}
\begin{proof}
We assume first that $G = F_m$ is a free group on $g_1, \dotsc, g_m$.
We first find words $w_1, \dotsc, w_{n^2} \in F_m$ such that 
$(\Delta(w_1), \dotsc, \Delta(w_{n^2}))$ is a basis of~$K^{n \times n}$.
Let $W_i := \{w \in F_m \mid |w| \leq i\}$, where $|w|$ denotes the length of 
the word $w$. For $X \subseteq K^{n \times n}$ denote by $\langle X \rangle_K$ 
the $K$-span of $X$.
Note that $\langle \Delta(W_{i+1}) \rangle_K = \langle \Delta(W_i) \rangle_K$ 
for some $i$ implies 
$\langle \Delta(W_j) \rangle_K = \langle \Delta(W_i) \rangle_K$ for all 
$j \geq i$. In particular, the chain
\[
    \langle \Delta(W_0) \rangle_K \subseteq \langle \Delta(W_1) \rangle_K 
        \subseteq \dotsb
\]
stabilizes after at most $n^2$ steps, so $\Delta(W_{n^2 - 1})$ is a generating 
set of $K^{n \times n}$.
Let $C$ be a subset of $W_{n^2-1}$ of $n^2$ elements; define the matrix 
$\Sigma := (\chi(v,w))_{v,w}$, where $v$ and $w$ run through $C$. Since the 
trace bilinear form 
$S\colon K^{n \times n} \times K^{n \times n} \to K\colon 
    (V,W) \mapsto \tr(VW)$
is non-degenerate, $\Delta(C)$ is a basis of $K^{n \times n}$ if and only if 
$\Sigma$ is non-singular. By running through all $n^2$-element subsets of 
$W_{n^2-1}$ we can find the $w_1, \dotsc, w_{n^2}$.

Now let $V := K^{n \times 1}$ be the $KF_m$-module induced by $\Delta$.
We first construct the $KF_m$-module 
$V^n = V \oplus \dotsb \oplus V \equiv K^{n \times n}$.
To determine the action of $F_m$ on $K^{n \times n}$, it is enough to determine 
values $\lambda^i_{jk} \in K$ such that 
$\Delta(g_i)\Delta(w_j) = \sum_k \lambda^i_{jk} \Delta(w_k)$, where 
$1 \leq i \leq m$ and $1 \leq j, k \leq n^2$. Since $S$ is non-degenerate, each 
$\lambda^i_{jk}$ is uniquely determined by the $n^2$ equations 
\[
    \chi(g_iw_jw_\ell) = S(\Delta(g_i)\Delta(w_j), \Delta(w_\ell)) 
        = S(\sum_k \lambda^i_{jk}\Delta(w_k), \Delta(w_\ell)) 
        = \sum_{k = 1}^{n^2}\lambda^i_{jk} \chi(w_k\cdot w_\ell), 
\]
where $1 \leq \ell \leq n^2$.
By solving the linear equations, we can construct the $KF_m$-module 
$V^n \equiv K^{n \times n}$.

Let $\Gamma\colon F_m \to \GL(K^{n \times n})$ be the representation on 
$V^n \equiv K^{n \times n}$. We denote the extensions of $\Delta$ and $\Gamma$ 
to the group algebras again by $\Delta$ and $\Gamma$, respectively.
Let $v = (v_1, \dotsc, v_n) \in K^{n \times n}$, where the $v_i$ are the 
columns of $v$. Then $\Gamma(a)v = (\Delta(a)v_1, \dotsc, \Delta(a)v_n)$ for 
$a \in KF_m$. In particular, $\Gamma(a)$ and $\Delta(a)$ have the same minimal 
polynomial, and if $c \in K[x]$ is the characteristic polynomial 
of~$\Delta(a)$, then $c^n$ is the characteristic polynomial of $\Gamma(a)$.

We now use an adaptation of \cite{obrien} to find a simple factor of the 
$KF_m$-module $K^{n \times n}$.
If $K$ is finite, choose random elements $a \in KF_m$ until $\Gamma(a)$ has an 
eigenspace of dimension~$n$. Since the image of $\Delta$ is isomorphic 
to~$K^{n \times n}$, this terminates with high probability by a result of Holt 
and Rees (see \cite[Section~2.3]{holt}).
Set $L := K$, and let $\lambda \in L$ be an eigenvalue of $\Gamma(a)$ of 
multiplicity~$n$. If $K$ is infinite, then choose random $a \in KF_m$ until the 
characteristic polynomial of $\Gamma(a)$ is an $n$-th power of a separable 
polynomial (that is, the characteristic polynomial of $\Delta(a)$ is 
separable). The characteristic polynomial of a matrix is inseparable if and 
only if its discriminant is zero, so the set of matrices with inseparable 
characteristic polynomial is Zariski closed in~$K^{n \times n}$.
Thus the matrices with separable characteristic polynomial are Zariski dense 
in~$K^{n \times n}$. Since the image of $\Delta$ is isomorphic 
to~$K^{n \times n}$, the probability of finding a suitable $a$ is very high.
Let $L/K$ be a field extension such that the characteristic polynomial has a 
root $\lambda$ in~$L$.

Let $v \in L^{n \times n}$ be an eigenvector of $\Gamma(a)$ with 
eigenvalue~$\lambda$. Then 
\[
    \Gamma(a)v = (\Delta(a)v_1,\dotsc, \Delta(a)v_n) 
        = \lambda v = (\lambda v_1, \dotsc, \lambda v_n).
\]
We may assume without loss of generality that $v_1$ is non-zero. Since the 
$\lambda$-eigenspace of $\Delta(a)$ is one-dimensional, there exist 
$\xi_2, \dotsc, \xi_n \in L$ such that $v_i = \xi_i v_1$ for $i > 1$.
Thus $v = (v_1, \xi_2 v_1, \dotsc, \xi_n v_1)$ and 
$\Gamma(a)v = (\Delta(a)v_1, \xi_2\Delta(a)v_1, \dotsc, \xi_n \Delta(a)v_1)$,
so $LF_mv$ is isomorphic to $LF_mv_1 \cong L \otimes_K V$. Now choose 
$w_1, \dotsc, w_n \in F_m$ such that 
$B := (\Gamma(w_1)v, \dotsc, \Gamma(w_n)v)$ is a basis of~$LF_mv$.
For every generator $g_i$ of $F_m$ let $\Delta'(g_i)$ be the representation 
matrix of $g_i$ on $LF_mv$ with respect to~$B$. By construction, $\Delta'$ is 
equivalent to~$\Delta$.
This concludes the proof if $G = F_m$ is a free group.

Now assume that $G$ is an arbitrary finitely generated group generated by $m$ 
elements, and let $\nu\colon F_m \to G$ be an epimorphism.
Let $\wh{\Delta} := \Delta \circ \nu$ and $\wh{\chi} := \chi \circ \nu$.
We construct an extension field $L/K$ and a representation $\wh{\Delta}'$
such that $\wh{\Delta} \sim \wh{\Delta}'$.
But then $\Delta'\colon G \to \GL(n, F)$ defined by 
$\Delta'(g) := \wh{\Delta}'(\wt{g})$, where $\wt{g} \in F_m$ with 
$\nu(\wt{g}) = g$ is arbitrary, is a representation of $G$, equivalent 
to~$\Delta$.
\end{proof}

In our special setting, we can use the trace polynomials to compute all 
character values. Furthermore, we always assume that 
$\Delta_{\langle g_1, g_2 \rangle}$ is absolutely irreducible, so we can choose 
$(w_1, \dotsc, w_4) = (1, g_1, g_2, g_1g_2)$ in the first part of the 
algorithm.

\section{Examples}
\label{S:examples}

For the results in this section we use our implementation of the 
$\Ln_2$-quotient algorithm in \textsc{Magma} \cite{magma}.

\subsection{Groups with finitely many $\Ln_2$-quotients}

In \cite{coxeter}, Coxeter defines three families of presentations:
\begin{align*}
    (\ell,m | n,k) & = \langle a,b \mid 
        a^\ell, b^m, (ab)^n, (a^{-1}b)^k \rangle, \\
    (\ell,m,n;q) & = \langle a,b \mid 
        a^\ell, b^m, (ab)^n, [a,b]^q \rangle, \\
    G^{m,n,p} & = \langle a,b \mid 
        a^m, b^n, c^p, (ab)^2, (ac)^2, (bc)^2, (abc)^2 \rangle.
\end{align*}
These groups have been intensively studied, see~\cite{edjvet} for an overview.
After recent work of Havas and Holt \cite{havas_holt}, only for four of these 
groups is it not known whether they are finite or infinite, namely 
$(3,4,9;2)$, $(3,4,11;2)$, $(3,5,6;2)$, and $G^{3,7,19}$. We study these groups 
and their low-index subgroups \cite{sims} using the $\Ln_2$-quotient algorithm.

\begin{prop}
Let $G = (3,4,9;2)$. Then $G$ has seven conjugacy classes of subgroups of 
index~$\leq 50$. For $1 \leq i \leq 50$ let $H_i \leq G$ with $[G:H_i] = i$, if 
such a group exists. The only $\Ln_2$-quotient of $H_i$ for 
$i \in \{1,3,4,12\}$ is $\Ln_2(89)$; the group $H_6$ has a quotient 
$\Ln_2(89)\times (\PGL(2,5) \Yup^{\Cyc_2} \PGL(2,5))$; and $H_{30}$ and 
$H_{36}$ have a quotient $\Ln_2(89)\times \PGL(2,5)$.

Let $G = (3,5,6;2)$. Then $G$ has two conjugacy classes of subgroups of 
index~$\leq 50$, a group of index~$3$ and $G$ itself. Both groups have the 
single $\Ln_2$-quotient $\Ln_2(61)$.

The groups $(3,4,11;2)$ and $G^{3,7,19}$ do not have non-trivial subgroups of 
index~$\leq 50$. Both groups have a single $\Ln_2$-quotient, namely 
$(3,4,11;2)$ has $\Ln_2(769)$, and $G^{3,7,19}$ has $\Ln_2(113)$.
\end{prop}

The next result concerns a question of Conder \cite{conder}, asking whether a 
group has non-trivial finite quotients.

\begin{prop}
The group
\begin{align*}
    G = \langle A,B,C,D,E,F \mid 
        & A^3, B^3, C^2, D^2, E^2, F^2, (AC)^3, (AD)^3, (AE)^3, (AF)^3, \\
    & (BC)^3, (BD)^3, (BE)^3, (BF)^3, (ABA^{-1}C)^2, (ABA^{-1}D)^2, 
        (A^{-1}BAE)^2, \\
    & (A^{-1}BAF)^2, (BAB^{-1}C)^2, (B^{-1}ABD)^2, (BAB^{-1}E)^2, 
        (B^{-1}ABF)^2 \rangle
\end{align*}
has no quotients isomorphic to $\Ln_2(q)$ or $\PGL(2,q)$ for any prime 
power~$q$.
\end{prop}

\subsection{Groups with $\Ln_2$-ideals of type $\Ln_2(\infty^k)$}

If the algorithm returns an ideal of type~$\Ln_2(\infty^k)$, then the group has 
infinitely many $\Ln_2$-quotients, finitely many in every characteristic.
Using algebraic number theory, the precise quotient types can be determined as 
already outlined in \cite[Example~8.1]{l2q}.
We illustrate the process by relaxing the conditions of the Coxeter 
presentation $G^{3,7,19}$.

\begin{prop}
Let $G = \langle a,b,c \mid a^3, b^7, (ab)^2, (ac)^2, (bc)^2, (abc)^2 \rangle$.
Then $G$ has finitely many $\Ln_2$-quotients in every characteristic.

More precisely, let $K/\Q$ be the splitting field of $X^6 - 4X^4 + 3X^2 + 1$ 
with Galois group 
$\Gamma = \Gal(K/\Q) \cong \langle (1,4), (1,2,3)(4,5,6) \rangle 
    = \Cyc_2 \wr \Cyc_3$.
For a prime $p \neq 2,7$ denote by $\varphi_p \in \Gamma$ the Frobenius 
automorphism mod~$p$. The $\Ln_2$-quotient in characteristic~$p$ is
\begin{enumerate}
    \item
    $\Ln_2(p)^3$ if $\varphi_p = ()$;

    \item
    $\Ln_2(p)^2 \times \PGL(2,p)$ if $\varphi_p \sim (1,4)$;

    \item
    $\Ln_2(p) \times \PGL(2,p) \Yup^{\Cyc_2} \PGL(2,p)$ if 
    $\varphi_p \sim (1,4)(2,5)$;

    \item
    $\PGL(2,p) \Yup^{\Cyc_2} \PGL(2,p) \Yup^{\Cyc_2} \PGL(2,p)$ if 
    $\varphi_p \sim (1,4)(2,5)(3,6)$;

    \item
    $\Ln_2(p^3)$ if $\varphi_p \sim (1,2,3)(4,5,6)^{\pm 1}$;

    \item
    $\PGL(2,p^3)$ if $\varphi_p \sim (1,2,3,4,5,6)^{\pm 1}$;
\end{enumerate}
Moreover, $G$ has quotients $\Ln_2(2^3)$ and $\PGL(2,7)$.
\end{prop}
In this case, we do not need the precise conjugacy type of the Frobenius 
automorphism, the decomposition of $X^6 - 4X^4 + 3x^2 + 1$ is enough.
For example, taking $p = 65537$, we see that 
$X^6 - 4X^4 + 3X^2 + 1 \in \F_p[X]$ has two irreducible factors of degree~$3$;
this shows that $G$ has a quotient $\Ln_2(65537^3)$. Taking $p = 8388617$ we 
see that $X^6 - 4X^4 + 3X^2 + 1 \in \F_p[X]$ has two linear factors and two 
factors of degree~$2$; this shows that $G$ has quotient 
$\Ln_2(8388617) \times \PGL(2, 8388617) \Yup^{\Cyc_2} \PGL(2, 8388617)$,
that is, there is precisely one $N \tri G$ with $G/N \cong \Ln_2(8388617)$,
precisely two $N \tri G$ with $G/N \cong \PGL(2,8388617)$, and no other 
$N \tri G$ with $G/N \cong \Ln_2(8388617^k)$ or $G/N \cong \PGL(2,8388617^k)$ 
for some $k \in \N$.
\begin{proof}
The algorithm returns the single $\Ln_2$-ideal $P = \langle 
x_{1} + 1,
x_{2}^3 + x_{2}^2 - 2x_{2} - 1,
x_{3}^2 + x_{2}^2 - 3,
x_{12},
x_{13},
x_{23},
x_{123}
\rangle$ of type $\Ln_2(\infty^6)$.
The zeroes are
\[
    t = (-1,-\xi^4 + 3\xi^2 - 1, \xi, 0,0,0,0) \in \F_q,
\]
where $\xi$ is a root of $X^6 - 4X^4 + 3X^2 + 1$.
We assume $\F_q = \F_p[\xi]$.
Let $\delta\colon G \to \PSL(2,q)$ be a realization of~$t$.
There is no characteristic such that $-\xi^4 + 3\xi^2 - 1 = 0$ or $\xi = 0$, 
so $\Delta$ is never imprimitive, by Proposition~\ref{P:dihedral}.
Furthermore, $\xi$ is never a root of $\Psi_k$ for $k \in \{3,4,5,6,8,10\}$,
so $|\delta(c)| > 5$ for all $q$ (see Remark~\ref{R:order}), hence the image of 
$\delta$ cannot be $\Alt_4$, $\Sym_4$, or $\Alt_5$.
Thus $\im(\delta) \in \{\Ln_2(q), \PGL(2,\sqrt{q})\}$.
The precise isomorphism type depends on the action of the Galois group.
Note that ${}^\alpha t = {}^\sigma t$ for a Galois automorphism $\alpha$ and a 
non-trivial sign system $\sigma$ if and only if $\sigma = (1,1,-1)$ and 
$\alpha(\xi) = -\xi$.
The result for $p \neq 2,7$ now follows by Corollary~\ref{C:bijectiontuples} 
and the fact that the Galois automorphism in characteristic~$p$ is determined 
by the Frobenius automorphism.
For $p \in \{2,7\}$ the result can be verified directly.
\end{proof}

\subsection{Groups with $\Ln_2$-ideals of type $\Ln_2(p^\infty)$}

The other kind of $\Ln_2$-ideals of Krull dimension~$1$ are the ones containing 
a prime~$p$. They seem to occur far less frequently in practice than ideals of 
type~$\Ln_2(\infty^k)$.
However, when they occur, we can again make precise statements about the 
quotients.

\begin{prop}
Let 
$G = \langle a,b,c \mid a^3 = 1, [a,c] = [c,a^{-1}], aba = bab, 
    abac^{-1} = caba \rangle$.
There exist epimorphisms $G \to \Ln_2(q)$ if and only if $q = 3^k$ for some 
$k \in \N$.
Similarly, there exist epimorphisms $G \to \PGL(2,q)$ if and only if $q = 3^k$ 
for some $k \in \N$.
\end{prop}
\begin{proof}
The algorithm returns the single $\Ln_2$-ideal
\[
    P = \langle 3, x_{1} + 1, x_{2} + 1, x_{12} - 1, x_{13} - x_{3}, 
        x_{23} - x_{3}, x_{123}^2 - x_{3}x_{123} + 1 \rangle
\]
of type $\Ln_2(3^\infty)$, so $\Ln_2$-quotients can only occur in 
characteristic~$3$, proving the `only if' parts. It remains to show that every 
$3$-power occurs. The zeroes of $P$ are the trace tuples of the form
\[
    t = (t_1, t_2, t_3, t_{12}, t_{13}, t_{23}, t_{123}) 
        = (2,2,\xi + \xi^{-1},1,\xi + \xi^{-1},\xi + \xi^{-1},\xi)
\]
with $\xi \in \ol{\F_3}$.
Let $k = [\F_3[\xi]:\F_3]$, and let $\delta\colon G \to \PSL(3, 3^k)$ be a 
realization of~$t$. If $k = 2\ell$ and $\xi^{3^\ell} = -\xi$, then the Galois 
automorphism $\alpha = (x \mapsto x^{3^\ell})$ and the sign system 
$\sigma = (1,1,-1)$ induce the same action on~$t$, so the image of $\delta$ is 
$\PGL(2, 3^\ell)$, by Proposition~\ref{P:PGL}.
Otherwise, the image is $\Ln_2(3^k)$.
\end{proof}

Variations of the presentation yield similar results. We omit the easy proof.

\begin{prop}
Let $H = \langle a,b,c \mid [a,c][a^{-1},c], [b,a]ba^{-1}, 
    a^{-1}c^{-1}abac^{-1}a^{-1}b^{-1} \rangle$.

\begin{enumerate}
    \item
    Let $G = H/\langle a^5 \rangle^H$.
    Then $\Ln_2(q)$ and $\PGL(2,q)$ are quotients of $G$ if and only if 
    $q = 5^k$ for some $k \in \N$.

    \item
    Let $G = H/\langle a^7, (ab^{-1})^8\rangle^H$.
    Then $\Ln_2(q)$ and $\PGL(2,q)$ are quotients of $G$ if and only if 
    $q = 7^k$ for some $k \in \N$.
    
    \item
    Let $G = H/ \langle a^{11}, (ab^{-1})^5\rangle^H$.
    Then $\Ln_2(q)$ and $\PGL(2,q)$ are quotients of $G$ if and only if 
    $q = 11^k$ for some $k \in \N$.

    \item
    Let $G = H /\langle a^{19}, (ab^{-1})^9\rangle^H$.
    Then $\Ln_2(q)$ is a quotient of $G$ if and only if $q = 19^k$ for some 
    $k \in \N$ or $q = 37$;
    and $\PGL(2,q)$ is a quotient of $G$ if and only if $q = 19^k$ for some 
    $k \in \N$.
\end{enumerate}
\end{prop}

\subsection{Coxeter groups}

\begin{example}
Let
\[
    C := \begin{pmatrix} 
        1 & 8 & 3 & 2 \\ 
        8 & 1 & 5 & 5 \\ 
        3 & 5 & 1 & 13 \\ 
        2 & 5 & 13 & 1 \end{pmatrix} \in \Z^{4 \times 4}.
\]
Then $\Ln_2(q)$ is a smooth quotient of $G_C$ if and only if $q$ is one of the 
five primes
\[
    79,\ 
    6\,449,\ 
    699\,127\,441,\ 
    8\,438\,303\,591\,453\,175\,937\,527\,551,\ 
    518\,103\,478\,579\,218\,726\,546\,844\,118\,197\,999.
\]
Similarly, $\PGL(2,q)$ is a smooth quotient of $G_C$ if and only if $q$ is one 
of the six primes
\[
    11\,311,\ 
    28\,081,\ 
    68\,466\,319,\ 
    24\,005\,442\,449,\ 
    13\,345\,982\,337\,089,\ 
    408\,327\,690\,683\,773\,678\,271.
\]
\end{example}

\begin{definition}
A \textit{C-group representation} of rank~$m$ is a pair $\mcC = (H, S)$ such 
that $S = \{a_1, \dotsc, a_m\}$ is a generating set of involutions of $H$ which 
satisfy the \textit{intersection property}
\[
    \langle a_i  \mid i \in I \rangle \cap \langle a_j  \mid j \in J \rangle 
        = \langle a_k  \mid k \in I \cap J \rangle \quad 
        \text{for all } I,J \subseteq \{1, \dotsc, m\}.
\]
\end{definition}

\begin{example}
Let
\[
    C := \begin{pmatrix} 
        1 & 2 & 3 & 3 \\ 
        2 & 1 & 3 & 4 \\ 
        3 & 3 & 1 & 2 \\ 
        3 & 4 & 2 & 1 \end{pmatrix} \in \Z^{4 \times 4}.
\]
Then $\Ln_2(7)$ is the only smooth quotient of $G_C$ of $\Ln_2$-type. 
A realization is given by
\[
    S = \left\{
    \begin{pmatrix} 0 & 1 \\ 6 & 0 \end{pmatrix},\ 
    \begin{pmatrix} 3 & 2 \\ 2 & 4 \end{pmatrix},\ 
    \begin{pmatrix} 2 & 6 \\ 5 & 5 \end{pmatrix},\ 
    \begin{pmatrix} 0 & 5 \\ 4 & 0 \end{pmatrix}\right\},
\]
and it is easy to check that this generating set satisfies the intersection 
property.
\end{example}

The intersection property can be checked easily if $G_C$ only has finitely many 
$\Ln_2$-quotients. Infinitely many quotients can be handled as well, but 
require a little more work, as shown in the following result.

\begin{prop}
The only finite group of $\Ln_2$-type having a C-group representation of 
rank~$4$ such that 
$(|a_1a_2|, |a_1a_3|, |a_1a_4|, |a_2a_3|, |a_3a_4|) = (2,3,2,3,3)$ is 
$\PGL(2,5)$.
\end{prop}
\begin{proof}
The algorithm returns only one $\Ln_2$-ideal $P$ of type $\Ln_2(\infty^2)$.
Let $P \subseteq M \tri \Phi_m'$ be a maximal ideal containing $P$, and let 
$t = t_M \in \F_q^{15}$ be the corresponding full trace tuple (see 
Theorem~\ref{T:maxspec}).
Let $H = \langle a_1, \dotsc, a_4 \rangle \leq \PSL(2, q)$ be the image of the 
induced projective representation.
Then
\begin{align*}
    t & = (t_1, t_2, t_3, t_4, t_{12}, \dotsc, t_{1234}) \\
    & = \left(0,0,0,0,0,-1,0,-1,\frac{2}{3},-1,\eta_4,\frac{4}{3}\eta_4,\eta_4,
        -\frac{2}{3}\eta_4,\frac{1}{3}\right),
\end{align*}
where $\eta_4^2 - 2 = 0$.
The induced trace tuple for $H_1 = \langle a_2, a_3, a_4 \rangle$ is
\[
    \theta := (t_2, t_3, t_4, t_{23}, t_{24}, t_{34}, t_{234}) 
    = \left(0,0,0,-1,\frac{2}{3},-1,-\frac{2}{3}\eta_4\right).
\]
We determine the isomorphism type of $H_{1}$.
By Proposition~\ref{P:dihedral}, $H_{1}$ is dihedral if and only if 
$t_{234} = 0$, that is, if and only if $2 \mid q$. The alternating group of 
degree~$4$ is not generated by involutions, and using the methods of 
Section~\ref{S:subgrouptests} it is easy to check that $H_{1} \cong \Sym_4$ if 
and only if $5 \mid q$; furthermore, $H_{1} \not\cong \Alt_5$ for all~$q$.  
So if $(q, 30) = 1$, then $H_{1}$ is of $\Ln_2$-type. More precisely, if 
$X^2 - 2$ has a solution mod~$p$, then $\eta_4 \in \F_p$, so 
$H_{1} = \PSL(2, p)$. If $X^2 - 2$ has no solution mod~$p$, then $\eta_4$ is a 
generator of $\F_{p^2}/\F_p$, and the Galois group acts by the automorphism 
$\alpha$ which maps $\eta_4$ to $-\eta_4$.
In particular, ${}^\alpha \theta = {}^\sigma\theta$ for $\sigma = (-1,-1,-1)$,
so $H_{1} = \PGL(2,p)$ by Proposition~\ref{P:PGL}.

We now determine the isomorphism type of $H$.
If $2 \mid q$, then $H$ is dihedral; in fact, in this case $\eta_4 = 0$, so 
$t \in \F_2^{15}$, that is, $H = \PSL(2,2) = \Sym_3 = H_{1}$.
If $X^2 - 2$ has a solution mod~$p$, then $H = \PSL(2,p)$; otherwise, 
${}^\alpha t = {}^s t$ for $s = (-1,-1,-1,-1)$ with $\alpha$ as above,
hence $H = \PGL(2,p)$.
In any case, unless $5\mid q$ we see $H = H_{1}$, so the generating set does 
not satisfy the intersection property. We compute the realization
\[
    A = \left(
        \begin{pmatrix} 3 & 0 \\ 0 & 2 \end{pmatrix},
        \begin{pmatrix} 0 & 4 \\ 1 & 0 \end{pmatrix},
        \begin{pmatrix} 4 & 4\eta_4 + 2 \\ 4\eta_4 + 3 & 1 \end{pmatrix},
        \begin{pmatrix} 0 & 2\eta_4 + 2 \\ 2\eta_4 + 3 & 0 \end{pmatrix}
    \right) \in \SL(2, 5^2)^4
\]
of the unique trace tuple in characteristic~$5$, and it is easy to check that 
the induced projective tuple satisfies the intersection property.
\end{proof}

In this way, the $\Ln_2$-quotient algorithm can be used in the classification 
of all C-group representations of $\Ln_2(q)$ and $\PGL(2,q)$ of rank~$4$ 
(\cite{cgroups}).

\section{Acknowledgments}
I thank Eamonn O'Brien for comments on an early version of the paper.

\bibliography{paper}

\noindent
Department of Mathematics \\
The University of Auckland \\
Prive Bag 92019 \\
Auckland\\
New Zealand \\
E-mal address: \texttt{s.jambor@auckland.ac.nz}
\end{document}